\documentclass[11pt,a4paper]{amsart}
\usepackage{amssymb,cleveref}
\usepackage[top=30mm,right=30mm,bottom=30mm,left=30mm]{geometry}
\newcommand{\PSp}{\mathop{\mathrm{PSp}}}
\newcommand{\PG}{\mathop{\mathrm{PG}}}
\newcommand{\AGL}{\mathop{\mathrm{AGL}}}
\newcommand{\ASL}{\mathop{\mathrm{ASL}}}
\newcommand{\Sp}{\mathop{\mathrm{Sp}}}
\newcommand{\SO}{\mathop{\mathrm{SO}}}
\newcommand{\SL}{\mathop{\mathrm{SL}}}
\newcommand{\PGL}{\mathop{\mathrm{PGL}}}
\newcommand{\PSU}{\mathop{\mathrm{PSU}}}

\newcommand{\PGO}{\mathop{\mathrm{PGO}}}
\newcommand{\GO}{\mathop{\mathrm{GO}}}
\newcommand{\PSL}{\mathrm{\mathrm{PSL}}}
\newcommand{\GU}{\mathop{\mathrm{GU}}}
\newcommand{\GL}{\mathop{\mathrm{GL}}}
\newcommand{\GG}{\mathop{\mathrm{G}}}
\newcommand{\FF}{\mathop{\mathrm{F}}}
\newcommand{\EE}{\mathop{\mathrm{E}}}
\newcommand{\CC}{\mathop{\mathrm{C}}}
\renewcommand{\AA}{\mathop{\mathrm{A}}}
\newcommand{\BB}{\mathop{\mathrm{B}}}
\newcommand{\DD}{\mathop{\mathrm{D}}}
\newcommand{\CSp}{\mathop{\mathrm{CSp}}}
\newcommand{\PCSp}{\mathop{\mathrm{PCSp}}}
\newcommand{\POmega}{\mathop{\mathrm{P}\Omega}}
\newcommand{\PGU}{\mathop{\mathrm{PGU}}}

\newcommand{\Aut}{\mathop{\mathrm{Aut}}}
\newcommand{\Out}{\mathop{\mathrm{Out}}}
\newcommand{\lcm}{\mathop{\mathrm{lcm}}}

\newcommand{\Alt}{\mathop{\mathrm{Alt}}}
\newcommand{\soc}{\mathop{\mathrm{soc}}}
\newcommand{\Sym}{\mathop{\mathrm{Sym}}}
\newcommand{\meo}{\mathop{\mathrm{meo}}}

\newcommand{\smax}{s_{\mathrm{max}}}
\newcommand{\umax}{u_{\mathrm{max}}}

\renewcommand{\wr}{\mathop{\mathrm{wr}}}
\newtheorem{theorem}{Theorem}[section]
\def\cent#1#2{{\bf C}_{#1}(#2)}
\def\Z#1{{\bf Z}(#1)}
\def\norm#1#2{{\bf N}_{#1}(#2)}
\newtheorem{corollary}[theorem]{Corollary}
\newtheorem{proposition}[theorem]{Proposition}
\newtheorem{lemma}[theorem]{Lemma}
\newtheorem{remark}[theorem]{Remark}
\newtheorem{notation}[theorem]{Notation}
\begin{document}

\title[Maximal element order]{On the maximum orders of elements of finite almost simple groups and primitive permutation groups}  

\author[S.~Guest]{Simon Guest}
\address{Simon Guest, Centre for Mathematics of Symmetry and Computation,
School of Mathematics and Statistics,
The University of Western Australia,
 Crawley, WA 6009, Australia\newline
 Current address: Mathematics, University of Southampton, Highfield, SO17 1BJ, United Kingdom}\email{s.d.guest@soton.ac.uk}

\author[J.~Morris]{Joy Morris} 
\address{Joy Morris, Department of Mathematics and Computer Science,
University of Lethbridge,
Lethbridge, AB. T1K 3M4. Canada}
\email{joy@cs.uleth.ca}

\author[C.~E.~Praeger]{Cheryl E. Praeger}
\address{Cheryl E. Praeger, Centre for Mathematics of Symmetry and Computation,
School of Mathematics and Statistics,
The University of Western Australia,
 Crawley, WA 6009, Australia\newline Also affiliated with King Abdulazziz 
University, Jeddah, Saudi Arabia}\email{Cheryl.Praeger@uwa.edu.au}

\author[P.~Spiga]{Pablo Spiga}
\address{Pablo Spiga,
Dipartimento di Matematica e Applicazioni, University of Milano-Bicocca,\newline
Via Cozzi 53, 20125 Milano, Italy}\email{pablo.spiga@unimib.it}

\thanks{Address correspondence to P. Spiga,
E-mail: pablo.spiga@unimib.it\\ 
The second author is supported in part by the National Science
  and Engineering Research Council of Canada.
The research is supported in part by the Australian Research Council grants  
FF0776186, and DP130100106.}


\begin{abstract}
We determine upper bounds for the maximum order of an 
element of a finite almost simple group with socle $T$ in terms of the 
minimum index $m(T)$ of a maximal subgroup of $T$: 
for $T$ not an alternating group we prove that, with finitely many exceptions, 
the maximum element order is at most $m(T)$. Moreover, apart from an explicit list 
of groups,
the bound can be reduced to $m(T)/4$. These results are applied to determine all 
primitive permutation groups on a set of size $n$ that contain permutations of order greater than or equal to $n/4$. 
\end{abstract}
\maketitle

\section{Introduction}\label{introduction}

In 1903, Edmund Landau~\cite{La1,La2} proved that the maximum order of an element of the 
symmetric group $\Sym(n)$ or alternating group $\Alt(n)$ of degree $n$ is $e^{(1+o(1))(n\log n)^{1/2}}$, 
though it is now known from work of 
Erd\"os and Turan~\cite{ET1,ET2} that most elements have
far smaller orders, namely at most $n^{(1/2+o(1))\log n}$ (see also~\cite{BT, BLNPS}).  
Both of these bounds compare the element orders with the parameter $n$, which is the least degree of a 
faithful permutation representation of $\Sym(n)$ or $\Alt(n)$.
Here we investigate this problem for all finite almost simple groups: 
\begin{center}
\emph{Find upper bounds for the maximum
element order of an almost simple group with socle $T$ in terms of the \emph{minimum degree} $m(T)$
of a faithful permutation representation of $T$.    }
\end{center}
We discover that the alternating and symmetric groups are 
exceptional with regard to this element order comparison. 
We also study maximal element orders for many natural classes of subgroups of $\Sym(n)$, in particular for many 
families of primitive subgroups. 
Our most general result for almost simple groups is Theorem~\ref{playultimate}. 
For a group $G$ we denote by $\meo(G)$ the maximum order of an element of $G$. 
We note that the value of $\meo(T)$ for 
$T$ a simple classical group of odd characteristic was determined in~\cite{KS} and its relation to $m(T)$ can be deduced.
If $G$ is almost simple, say $T\leq G\leq {\rm Aut}(T)$ with its socle $T$ a non-abelian simple group, then
naturally $\meo(G)\leq\meo({\rm Aut}(T))$.

\begin{theorem}\label{playultimate}
Let $G$ be a finite almost simple group with socle $T$, such that $T\ne \Alt(m)$ for any
$m\geq5$. Then with finitely many exceptions, $\meo(G)\leq m(T)$; and indeed either $T=\PSL_d(q)$ for some $d,q$, or
$\meo(G)\leq m(T)^{3/4}$. Moreover, given 
positive $\epsilon,A >0$, there exists $Q=Q(\epsilon, A)$ such that, if $T=\PSU_4(q)$ with $q>Q$, 
then $\meo(G)> A\,m(T)^{3/4-\epsilon}$.
\end{theorem}

We note again that this result gives upper bounds for $\meo({\rm Aut}(T))$ in terms of $m(T)$,
and for $\meo(G)$ in terms of $m(G)$ (since $m(T)\leq m(G)$).   
Moreover 
equality in the upper bound $\meo({\rm Aut}(T))\leq m(T)$ holds when $T=\PSL_d(q)$ for all but two pairs $(d, q)$,
see Table~\ref{newtable} and Theorem~\ref{cor:PGammaL}. (Theorem~\ref{cor:PGammaL} and 
Table~\ref{newtable} provide good estimates for $\meo({\rm Aut}(T))$ for all finite classical simple groups $T$ in terms
of the field size and dimension.)
We are particularly interested in linear upper bounds for $\meo({\rm Aut}(T))$ of the form $c\, m(T)$  
with a constant $c<1$. It turns out that, after excluding the groups $\Alt(m)$ and $\PSL_d(q)$, such an upper 
bound holds with the constant $c=1/4$ for all but $12$ simple groups $T$.

\begin{theorem}\label{main2}
For a finite non-abelian simple group $T$, 
either $\meo(\Aut(T))< m(T)/4$, or $T$ is listed in Table~$\ref{exceptionsmain2}$. 
\end{theorem}

\begin{table}[!h]
\begin{tabular}{|ll|l|l|l|l|}\hline
$M_{11}$&$M_{23}$&$\Alt(m)$&$\PSL_d(q)$&$\PSU_3(3)$&$\PSp_6(2)$\\
         $M_{12}$&$M_{24}$&&           &$\PSU_3(5)$&$\PSp_8(2)$\\
         $M_{22}$&$HS$&&           &$\PSU_4(3)$&$\PSp_4(3)$\\\hline
\end{tabular}
\caption{Exceptions in Theorem~\ref{main2}}\label{exceptionsmain2}
\end{table}

Clearly, Theorems~\ref{playultimate} and~\ref{main2} do not provide the last word on this type of result. One might wonder, if minded so, ``What is the slowest growing function of $m(T)$ with the property that Theorem~\ref{main2} is still valid?'' (possibly allowing a \emph{finite} extension of the list in Table~\ref{exceptionsmain2}). We do not investigate this here. Instead we turn our attention to $\meo(G)$ for a wider family of primitive permutation groups $G$ than the almost simple primitive groups. For such groups of degree $n$, it also turns out that $\meo(G)<n/4$, apart from a number of explicitly determined families and individual primitive groups. We refer to~\cite{GMPSaffine} for the affine case in which $G$ has an abelian socle, since the proof in that case is very delicate and quite different from the arguments in this paper, which are based on properties of finite simple groups.

\begin{theorem}\label{main}
Let $G$ be a finite primitive permutation group of degree $n$ such that $\meo(G)$ is at least $n/4$. 
Then the socle $N\cong T^\ell$ of $G$ is isomorphic to  one of the following (where $k, \ell\geq1$):
\begin{enumerate}
\item\label{1}$\Alt(m)^\ell$ in its natural action on $\ell$-tuples of $k$-subsets from
$\{1,\ldots,m\}$; 
\item\label{2} $\PSL_d(q)^\ell$ in either of its natural actions on $\ell$-tuples of points, or $\ell$-tuples of hyperplanes,
of the projective space $\PG_{d-1}(q)$; 
\item\label{3} an elementary abelian group  $C_p^\ell$ and $G$ is described in~\cite{GMPSaffine}; or to
\item\label{4} one of the groups in Table~$\ref{exceptions}$.
\end{enumerate}
Moreover,  there exists a positive integer $\ell_T$, depending only on $T$, such that  $\ell\leq \ell_T$. 
\end{theorem}

\begin{remark} 
 \emph{ The possibilities for the degree $n$ of $G$ in Theorem~\ref{main}(4) are, in fact, quite restricted. In column 2 of Table \ref{anothertable}, we list the possibilities for the degree $m$ of the permutation representation of the socle factor $T$ of a primitive group $G$ of PA type of degree $n=m^\ell$. The integer $\ell$ can be as small as $1$, in which case $G$ is of AS type, and has maximum value $\ell_T$, which is also listed in column 2. If $G$ is of HS or SD type (with socle $\mathrm{Alt}(5)^2$) then we simply have $n=60$. } 
\end{remark}

Our choice 
of $n/4$ in  Theorems~\ref{main2} and~\ref{main}  is 
in some sense arbitrary. However it yields a list of exceptions that is not too cumbersome to obtain and 
to use, and yet is  sufficient to provide useful information on the normal covering number of $\Sym(m)$, an application 
described in~\cite{GMPScycles}. (The normal covering number of a non-cyclic group $G$ is the smallest number of 
conjugacy classes of proper subgroups of $G$ such that the union of the subgroups in all of these conjugacy classes 
is equal to $G$, that is to say the classes `cover' $G$.) In~\cite{GMPScycles} we use Theorem~\ref{main} to study primitive permutation groups
containing elements with at most four cycles, and our results about such groups yield critical information on normal covers
of $\Sym(n)$, and a consequent number theoretic application.

\subsection{Comments on the proof of Theorem~\ref{main}}\label{sub1}
Our proof of Theorem~\ref{main} uses the bounds of Theorem~\ref{main2}, and proceeds
according to the structure of $G$ and its socle as specified by the ``O'Nan--Scott type'' of $G$.  
This is one of the most effective modern methods for analysing 
finite  primitive permutation groups. The \emph{socle} $N$ of
$G$ is the subgroup generated by the minimal normal subgroups
of $G$. For an arbitrary finite group the socle is isomorphic to a direct product 
of simple groups, and, for finite primitive
groups these simple groups are pairwise isomorphic. The O'Nan--Scott
theorem describes in detail the embedding of $N$ in $G$ and provides
some useful information on the action of $N$, identifying a small number of 
pairwise disjoint possibilities. The subdivision we use in our proofs is described in~\cite{Pr} where eight types
of primitive groups are defined (depending on the structure and on the
action of the socle), namely HA (\textit{Holomorphic Abelian}), AS
(\textit{Almost Simple}), SD
(\textit{Simple Diagonal}), CD
(\textit{Compound Diagonal}), HS
(\textit{Holomorphic Simple}), HC
(\textit{Holomorphic Compound}), TW 
(\textit{Twisted wreath}), PA
(\textit{Product Action}), and it follows from the O'Nan--Scott Theorem (see~\cite{LPS1} or~\cite[Chapter 4]{DM}) 
that every primitive group is of exactly one of these types. 

In the light of this subdivision, Theorem~\ref{main} asserts that a finite primitive group containing elements of  
large order relative to the degree is either of AS or PA type  (with a well-understood socle), or of HA type, 
or it has bounded order. The proof of Theorem~\ref{main} for primitive groups of HA type is in our companion 
paper~\cite{GMPSaffine}, where we obtain an explicit description of the permutations $g\in G$ with order $|g|\geq n/4$ 
together with detailed information on the structure of $G$. We refer the interested reader to~\cite{GMPSaffine} 
for more information on this case.

\begin{table}
\begin{tabular}{|l|l|ll|l|l|l|l|}\hline
\multicolumn{6}{|c|}{AS type}&HS or SD &PA type\\
\multicolumn{6}{|c|}{}&type&\\
\hline
$\Alt(5)$&$M_{11}$&$\PSL_2(7)$ &$\PSL_2(49)$&$\PSU_3(3)$&$\PSp_6(2)$&$\Alt(5)^2$&$T^\ell$ where\\
$\Alt(6)$&$M_{12}$&$\PSL_2(8)$ &$\PSL_3(3)$ &$\PSU_3(5)$&$\PSp_8(2)$&&$T$ is one of\\
$\Alt(7)$&$M_{22}$&$\PSL_2(11)$ &$\PSL_3(4)$&$\PSU_4(3)$&$\PSp_4(3)$&&the groups\\
$\Alt(8)$&$M_{23}$&$\PSL_2(16)$ &$\PSL_4(3)$ &&&&in the AS type\\
$\Alt(9)$&$M_{24}$&$\PSL_2(19)$&&&&&part of\\
         &$HS$   &$\PSL_2(25)$&  &&&&this table\\\hline
\end{tabular}
\caption{The socles for the exceptions $G$ in Theorem~\ref{main}~\eqref{4}}
\label{exceptions}
\end{table}

\subsection{Structure of the paper}\label{sub2}
In Section~\ref{prel} we determine tight upper bounds on the maximum element orders for the almost simple groups and we 
give in Table~\ref{newtable} some valuable information on the maximum element order of $\Aut(T)$ when $T$ is a simple 
group of Lie type. In Section~\ref{md}, we collect some well-established results on the minimal degree of a permutation 
representation for the non-abelian simple groups. (These include corrections noticed by Mazurov and Vasil$'$ev~\cite{Vav4} 
to~\cite[Table 5.2.A]{KL}.) We then prove Theorem~\ref{main2} in Section~\ref{meo}. The proof of 
Theorem~\ref{main}, which relies on Theorem~\ref{main2}, is given in Section~\ref{sec:them2}. We provide some information 
on the positive integers $\ell_T$ (defined in Theorem~\ref{main2})  in Remark~\ref{rm-1} and in Table~\ref{anothertable}. 
Finally, Section~\ref{sec:playultimate} contains the proof of Theorem~\ref{playultimate}.

\section{Maximum element orders for simple groups}\label{prel}

For a finite group $G$, we write $\exp (G)$ for the \emph{exponent} of $G$; that is, 
the minimum positive integer $k$ for which $g^k=1$ for all $g\in G$. We denote the 
\emph{order} of the element $g\in G$ by $|g|$ and we write $\meo(G)$ for the 
\emph{maximum element order} of $G$; that is, $\meo(G)=\max\{|g|\mid g\in G\}$. 
Clearly, $\meo(G)$ divides $\exp(G)$. 

In this section we study $\meo(G)$ where $G$ is  an almost simple group. We start by considering the symmetric groups. 
It is well-known that 
\[ \meo(\Sym(m)) = \max \{\lcm(n_1, \ldots ,n_N) \mid m=n_1+\cdots+ n_N\}. \] 
The expression $\meo(\Sym(m))$ is often referred to as  \emph{Landau's function} (and is usually denoted by $g(m)$),
in honour of Landau's theorem in~\cite{La1}.  We record  the main results from~\cite{La1} and~\cite{Mass} on $\meo(\Sym(m))$, to which we will refer  in the sequel. As usual $\log (m)$ denotes the logarithm of $m$ to the base $e$.

\begin{theorem}[{\cite{La1} and~\cite[Theorem~$2$]{Mass}}]\label{Landau}
For all $m\geq 3$, we have \[ \sqrt{m\log (m)/4}\leq \log (\meo(\Sym(m)))\leq \sqrt{m\log m}\left(1+\frac{\log(\log (m))-a}{2\log (m)}\right) \]
with $a=0.975$.
\end{theorem}
\begin{proof}The lower bound is proved in~\cite{La1} and the upper bound is proved in~\cite{Mass}.
\end{proof}

Since $\Aut(\Alt(m))\cong\Sym(m)$ unless $m \in \{2,6\}$, Theorem~\ref{Landau} 
gives good estimates of the maximum element 
order of $\Aut(\Alt(m))$. And since the minimal degree of a permutation representation of $\Alt(m)$ is $m$, for $m\ne6$, 
we find that $\Alt(m)$ is one of the exceptional groups in Theorem~\ref{main2} listed in Table~\ref{exceptionsmain2}.

For the groups of Lie type, the following three lemmas will be used frequently in the proof of Theorem~\ref{main2}. 
Here $\log_p(x)$ denotes the logarithm  of $x$ to the base $p$ and $\lceil x\rceil$ denotes the least integer $k$ satisfying 
$x\leq k$. We denote by $J_d$ the {\em cyclic unipotent element} of $\GL_d(q)$ that sends the canonical basis element $e_i$ to $e_i+e_{i+1}$ for $i< d$ and fixes $e_d$; that is, $J_d$ is a $d\times d$ unipotent Jordan block. Also, we denote the identity matrix in $\GL_d(q)$ by $I_d$.

\begin{lemma}\label{uni}
Let $u$ be a unipotent element of $\GL_d(p^{f})$ where $p$ is
prime. Then $|u|\leq p^{\lceil \log_p(d)\rceil}$ and equality holds if and only if the Jordan decomposition of $u$  has a block of size $b$ such that $\lceil\log_p(d)\rceil=\lceil\log_p(b)\rceil$.  
\end{lemma}
\begin{proof}
Let $b$ be the dimension of the largest Jordan block of $u$. Let $B=J_b-I_b$, 
a $b\times b$ matrix over $\mathbb{F}_{p^f}$. Then since $J_b$ is unipotent, 
it follows that $B$ is nilpotent and $B^b=0$. Now fix a positive integer 
$k$. Using the binomial theorem, we have 
\[
J_b^{p^{k}}=(I_b+B)^{p^k}=\sum_{i=0}^{p^k}\binom{p^k}{i}B^i.
\]
Since $\binom{p^k}{i}$ is divisible by $p$ for every $i\in
\{1,\ldots,p^k-1\}$, we have $J_b^{p^k}=I_b+B^{p^k}$. In particular, $J_b^{p^k} =I_b$ if and only if $B^{p^k}=0$. 
Since $J_b$ is a cyclic unipotent element, $b$ is the least positive integer such that $B^b=0$; therefore $r=\lceil 
\log_p(b)\rceil$ is the least nonnegative integer such that $B^{p^r}=0$. Thus $|J_b|=p^{\lceil \log_p(b)\rceil}$.

Suppose that the maximum size of a Jordan block of $u$ is $b$. Then by the previous paragraph, 
$|u|=|J_b|= p^{\lceil \log_p(b)\rceil}$.  Since $b\leq d$, this implies that $|u|\leq p^{\lceil \log_p(d)\rceil}$ and that equality holds if and only if $\lceil\log_p(d)\rceil=\lceil\log_p(b)\rceil$.  
\end{proof}

The following elementary lemma, on the direct product of cyclic groups, will be applied to the maximal tori of groups of Lie type.

\begin{lemma}\label{silly1}
Let $k$ be a positive integer, and for each $i\in \{1,\ldots,t\}$, let $k_i$ be a multiple of $k$ and let 
$C_{i}=\langle x_i\rangle$ be a cyclic group of order $k_i$. Let $C$ be the subgroup of $G:= C_1\times\cdots\times C_t$
of order $k$ generated by $x_1^{k_1/k}\cdots x_t^{k_t/k}$. Then the exponent of the quotient group $G/C$ is $k_1/k$ if 
$t=1$ and $\lcm\{k_1,\ldots,k_t\}$ if $t \ge 2$.
\end{lemma}
\begin{proof}
If $t=1$, then the exponent of $\langle x_1\rangle/\langle x_1^{k_1/k}\rangle$ is clearly $k_1/k$. 
So suppose that $t \ge 2$. Set $r=\lcm\{k_1,\ldots,k_t\}$ and $r'=\exp(G/C)$. The group $G$ has exponent $r$ and so 
$r'=\exp(G/C)\leq r$. Conversely, for each $i\in \{1,\ldots,t\}$, we have $x_i^{r'}\in C$. Since $t \ge 2$, we have $C_{i}\cap C=1$
because the non-trivial elements of $C$ all have the form $x_1^{jk_1/k}\cdots x_t^{jk_t/k}$ with $1\leq j<k$, and so do not lie in $C_i$. 
Thus $x_i^{r'}=1$. This shows that, for each $i\in \{1,\ldots,t\}$, the integer $k_i$ divides $r'$. Therefore $r\leq r'$, and so $r'=r$.
\end{proof}

The following technical lemma  will be applied repeatedly to estimate the maximum element order of a group of Lie type.

\begin{lemma}\label{silly2} Suppose that $m,k,f,p$ are positive integers where $p$ is prime and $q=p^f$. Then
\begin{enumerate}
\item[(i)]$q^k-1$ divides $q^{km}-1$ and $(q^{km}-1)/(q^k-1)\geq p^{\lceil\log_p(m)\rceil}$;

\item[(ii)]if $m$ is odd, then $q^{k}+1$ divides $q^{km}+1$; furthermore, if  $(p,k,m,f)\neq (2,1,3,1)$, then $(q^{km}+1)/(q^k+1)\geq p^{\lceil \log_p(m)\rceil}$;

\item[(iii)]if $m$ is even, then $q^k+1$ divides $q^{km}-1$; furthermore, if  $(k,m,f)\neq (1,2,1)$, then $(q^{km}-1)/(q^k+1)\geq  p^{\lceil\log_p(m)\rceil}$. 
\end{enumerate}
\end{lemma}
\begin{proof}
The divisibility assertions in~(i),~(ii) and~(iii) are obvious. For Part~(i), note that $(q^{km}-1)/(q^k-1)=q^{k(m-1)}
+q^{k(m-2)}+\cdots +q^k+1\geq q^{k(m-1)}$. Furthermore, $q^{k(m-1)}\geq q^{m-1}\geq p^{m-1}\geq m$ and so $m-1 \ge 
\log_p(m)$. However $m-1$ is an integer, so $m-1 \ge \lceil \log_p(m) \rceil$ and  $(q^{km}-1)/(q^k-1)\geq p^{m-1} 
\geq p^{\lceil\log_p(m)\rceil}$. 

Assume that $m$ is odd. The assertions hold if $m=1$, so assume that $m\geq3$. Then $(q^{km}+1)/(q^k+1)\geq q^{k(m-2)}= 
p^{fk(m-2)}\geq m$ (where the last inequality holds for $m\ge 3$ provided $(p,k,m,f)\neq (2,1,3,1)$). So, arguing as in 
the previous paragraph, we have $(q^{km}+1)/(q^k+1)\geq  p^{\lceil\log_p(m)\rceil}$ for $(p,k,m,f)\neq (2,1,3,1)$, which 
gives Part~(ii).

Next, suppose that $m$ is even. The assertions all hold for $m=2$ unless $(k,m,f)=(1,2,1)$. 
So assume that $m\geq 4$. Then $(q^{km}-1)/(q^k+1)\geq q^{k(m-2)}= p^{fk(m-2)}\geq m$. Now arguing as in the first paragraph 
we have $(q^{km}-1)/(q^k+1)\geq  p^{\lceil\log_p(m)\rceil}$, which proves Part~(iii).
\end{proof}

Before proceeding and obtaining some tight bounds on the maximum element order for the groups of Lie type, 
we need to prove some results on centralizers of semisimple elements in $\PGL_d(q)$ and related classical groups. 
In order to do so, we introduce some notation.

\begin{notation}\label{not1}{\rm 
Let $\delta=1$ unless we deal with a unitary group in which case let $\delta=2$.
Let $s$ be a semisimple element of $\PGL_d(q^\delta)$ and let $\overline{s}$ be a semisimple 
element of $\GL_d(q^\delta)$ projecting to $s$ in $\PGL_d(q^\delta)$. The action of the matrix 
$\overline{s}$ on the $d$-dimensional vector space $V=\mathbb{F}_{q^\delta}^d$ naturally 
defines the structure of an $\mathbb{F}_{q^\delta}\langle \overline{s}\rangle$-module on $V$.  
Since $\overline{s}$ is semisimple, $V$ decomposes, by Maschke's theorem, as a direct 
sum of irreducible $\mathbb{F}_{q^\delta}\langle \overline{s}\rangle$-modules, that is, 
$V=V_1\oplus\cdots\oplus V_{l}$, with $V_i$ an irreducible $\mathbb{F}_{q^\delta}\langle 
\overline{s}\rangle$-module. Relabelling the index set $\{1,\ldots,l\}$ if necessary, 
we may assume that the first $t$ submodules $V_1,\ldots,V_t$ are pairwise non-isomorphic 
(for some $t\in\{1,\ldots,l\}$) and that for $j\in\{t+1,\ldots,l\}$,  $V_j$ is isomorphic 
to some $V_i$ with $i\in\{1,\ldots,t\}$. Now, for $i\in \{1,\ldots,t\}$, let $\mathcal{W}_i
=\{W\leq V\mid W\cong V_i\}$, the set of $\mathbb{F}_{q^\delta}\langle \overline{s}\rangle$-submodules 
of $V$ isomorphic to $V_i$ and write $W_i=\sum_{W\in \mathcal{W}_i}W$. The module $W_i$ is 
usually referred to as the \emph{homogeneous} component of $V$ corresponding to the simple 
submodule $V_i$. We have $V=W_1\oplus \cdots \oplus W_t$. Set $a_i=\dim_{\mathbb{F}_{q^\delta}}(W_i)$. 
Since $V$ is completely reducible, we have  $W_i=V_{i,1}\oplus\cdots\oplus V_{i,m_i}$ for some 
$m_i\geq 1$, where $V_{i,j}\cong V_i$, for each $j\in \{1,\ldots,m_i\}$. Thus we have $a_i=
d_im_i$, where $d_i=\dim_{\mathbb{F}_{q^\delta}}V_i$, and $
\sum_{i=1}^{t}d_im_i=d$. For 
$i\in \{1,\ldots,t\}$, we let $x_i$ (respectively $y_{i,j}$) denote the element in $\GL(W_i)$ 
(respectively $\GL(V_{i,j})$) induced by the action of $\overline{s}$ on $W_i$ (respectively 
$V_{i,j}$). In particular, $x_i=y_{i,1}\cdots y_{i,m_i}$ and $\overline{s}=x_1\cdots x_t$.
We note further that 
\[ 
p(s)=(\underbrace{d_1,\ldots,d_1}_{m_1\textrm{ times}},\underbrace{d_2,\ldots,d_2}_{m_2\textrm{ times}},
\ldots,\underbrace{d_t,\ldots,d_t}_{m_t\textrm{ times}}) 
\]
is a partition of $n$. 

Now let $c \in \cent {\GL_d(q^\delta)}{\overline{s}}$. Given $i\in \{1,\ldots,t\}$ and 
$W\in \mathcal{W}_i$, we see that $W^c$ is an $\mathbb{F}_{q^\delta}\langle \overline{s}\rangle$-submodule 
of $V$ isomorphic to $W$ (because $c$ commutes with $\overline{s}$). Thus $W^c\in\mathcal{W}_i$. This shows 
that $W_i$ is $\cent{\GL_d(q^\delta)}{\overline{s}}$-invariant. It follows that 
\[ 
\cent{\GL_d(q^\delta)}{\overline{s}}=\cent{\GL(W_1)}{x_1}\times\cdots\times \cent{\GL(W_t)}{x_t} 
\] 
and every unipotent element of $\cent {\GL_d(q^\delta)}{\overline{s}}$ is of the 
form $u=u_1\cdots u_t$ with $u_i\in \cent{\GL(W_i)}{x_i}$ unipotent in $\GL(W_i)$, for each $i$. 

Since $\overline{s}$ is semisimple and $V_{i,j}$ is irreducible, Schur's lemma 
implies that $V_{i,j}\cong \mathbb{F}_{q^{\delta d_i}}$ and that the action of $y_{i,j}$ 
on $V_{i,j}$  is equivalent to the scalar multiplication action  on $\mathbb{F}_{q^{d_i}}$ by  
a field generator $\lambda_{i,j}$ of $\mathbb{F}_{q^{\delta d_i}}$. As $V_{i,j_1}\cong V_{i,j_2}$, we 
have $\lambda_{i,j_1}=\lambda_{i,j_2}$,  for $j_1,j_2\in\{1,\ldots,,m_i\}$ and  we write 
$\lambda_i=\lambda_{i,1}$. Under this identification, replacing 
$x_i$ by a suitable conjugate in $\GL_{a_i}(q^\delta)$ if necessary, we have $x_i=\lambda_i 
I_{m_i}\in\GL_{m_i}(q^{\delta d_i}) < \GL_{a_i}(q^\delta)$. 
Now a direct computation shows that $\cent{\GL(W_i)}{x_i}\cong \GL_{m_i}(q^{\delta d_i})$.
}
\end{notation}
  
\begin{proposition}\label{more refined}
Let $s$ be as in Notation~$\ref{not1}$. A unipotent element $u$ of $\PGL_d(q)$ centralizing $s$ has order at most $\max\{p^{\lceil \log_p(m_1)\rceil},\ldots,p^{\lceil \log_p(m_t)\rceil}\}$. 
\end{proposition}
\begin{proof}We use the notation established in Notation~\ref{not1}. Let $u$ be a unipotent element of $\PGL_d(q)$ and let 
$\overline{u}$ be the unique unipotent element of $\GL_d(q)$ projecting to $u$. 
Since $u$ centralizes $s$, $\overline{u}$ commutes with $\overline{s}$ modulo $\Z{\GL_d(q)}$. Thus $\overline{u}\,
\overline{s}=(\overline{s}\,\overline{u})c$, for some scalar matrix $c$ of $\GL_d(q)$. Arguing by induction, we see 
that, for each $k\geq 1$, we have $\overline{u}^k\overline{s}=\overline{s}\,\overline{u}^kc^k$. In particular, for $k=q-1$, 
since $c^{q-1}=1$, it follows that $\overline{u}^{q-1}$ centralizes $\overline{s}$. Since the order of $\overline{u}$ is a 
$p$-power, we find that $\overline{u}$ centralizes $\overline{s}$. Thus $|u|$ is bounded above by the maximum order a unipotent 
element in $\cent{\GL_d(q)}{\overline{s}}\cong \GL_{m_1}(q^{d_1})\times \cdots \times \GL_{m_t}(q^{d_t})$. The result now
follows from Lemma~\ref{uni}.
\end{proof}

The following corollary is well-known and somehow not surprising.
\begin{corollary}\label{cor:PGL}$\meo(\PGL_d(q))=(q^d-1)/(q-1)$.
\end{corollary}
\begin{proof}A Singer cycle of $\PGL_d(q)$ has order $(q^d-1)/(q-1)$ and so $\meo(\PGL_d(q))\geq (q^d-1)/(q-1)$.
Let $g\in\PGL_d(q)$. Then $g$ has a unique expression as $g=su=us$ with $s$ semisimple and $u$ unipotent. 
We use Notation~\ref{not1} for the element $s$. By Lemma~\ref{silly1} and the proof of Proposition~\ref{more refined}, 
we see that if $t=1$, so that $d=m_1d_1$, then
\[ 
|g|\leq \frac{q^{d_1}-1}{q-1}p^{\lceil\log_p(m_1)\rceil}\leq \frac{q^{d}-1}{q-1} 
\] 
(using Lemma~\ref{silly2}(i)). If $t\geq2$, then 
\[
|g|\leq  \lcm\{(q^{d_i}-1)p^{\lceil\log_p(m_i)\rceil }\mid i=1,\ldots,t \}
\le  \frac{1}{(q-1)^{t-1}}\prod_{i=1}^{t}(q^{d_i}-1)p^{\lceil\log_p(m_i)\rceil },
\]
which by Lemma~\ref{silly2}~(i) is at most 
\[  \frac{1}{(q-1)^{t-1}}\prod_{i=1}^{t}(q^{d_im_i}-1) \le \frac{q^{d}-1}{(q-1)^{t-1}}\leq \frac{q^d-1}{q-1}.
\]
\end{proof}

\begin{remark}{\rm 
As one might expect, sometimes we have $\meo(\Aut(\PSL_d(q)))>(q^d-1)/(q-1)$. For example, $\PGL_2(4)=\PSL_2(4)\cong \Alt(5)$ 
and $\meo(\PSL_2(4))=5$, but  $\Aut(\Alt(5))=\Sym(5)$ and $\meo(\Sym(5))=6$. Later, in Theorem~\ref{cor:PGammaL} 
(using an application of Lang's theorem) we will prove that, in fact, $\meo(\Aut(\PSL_d(q)))=(q^d-1)/(q-1)$ in all other cases.} 
\end{remark}

Before studying other classical groups we need the following number-theoretic lemma which will be crucial in 
studying  the asymptotic value of $\meo(\PSp_{2m}(q))$ as $m$ tends to infinity  
(see Corollary~\ref{cor:PSp} and Remark~\ref{rem:2}). In the proof of Lemma~\ref{tori}, we denote by $(a)_2$ the 
largest power of $2$ dividing the positive integer $a$.

\begin{lemma}\label{tori}Let $(a_1,\ldots,a_t)$ be a partition of
  $d$, let $q$ be a prime power and, for each $i\in \{1,\ldots,t\}$, let $\varepsilon_i\in\{-1,1\}$. Then
  $\lcm_{i=1}^t\{q^{a_i}-\varepsilon_i\}\leq q^{d+1}/(q-1)$ if $q$ is even or $t=1$, and 
$\lcm_{i=1}^t\{q^{a_i}-\varepsilon_i\}\leq q^{d+1}/2(q-1)$ if $q$ is odd and $t\geq2$.
\end{lemma}
\begin{proof}
Set $L:=\lcm_{i=1}^t\{q^{a_i}-\varepsilon_i\}$.
If $t=1$, then $L=q^{d}-\varepsilon_1\leq q^{d}+1=q^{d}(1+1/q^d)\leq q^{d+1}/(q-1)$ and the lemma is proved. 
Thus we may assume that $t>1$. We argue by induction on $d$. Write
$I=\{i\in\{1,\ldots,t\}\mid \varepsilon_i=-1\}$.  If 
$a_i=a_j$ for distinct elements $i, j\in I$ then, replacing $d$ by $d-a_j$ and replacing the partition
$(a_1,\ldots,a_t)$ by the same partition with the part $a_j$
removed, it follows by induction that $L\leq q^{d-a_j+1}/(q-1)\leq q^{d+1}/2(q-1)$. 
Therefore, we may assume further that the set $\{a_i\}_{i\in I}$
consists of pairwise distinct elements. Let $\alpha$ and $\beta$ be distinct elements of 
$\{1,\ldots,t\}$ and write $r=\gcd(q^{a_\alpha}-\varepsilon_\alpha,q^{a_\beta}-\varepsilon_\beta)$
and $s=(\gcd(q-1,2))^{t-1}$. Now

\begin{eqnarray}\label{eq1}\nonumber
L=\lcm_{i=1}^t\{q^{a_i}-\varepsilon_i\}&\leq&\frac{1}{rs}\prod_{i\in
  I}(q^{a_i}+1)\prod_{i\notin I}(q^{a_i}-1)\leq\frac{1}{rs} 
\prod_{i\in I}q^{a_i}\prod_{i\in  I}\left(1+\frac{1}{q^{a_i}}\right)\prod_{i\notin
  I}q^{a_i}\\
&=&\frac{q^d}{rs}\prod_{i\in  I}\left(1+\frac{1}{q^{a_i}}\right)\leq
\frac{q^d}{rs}\prod_{k\in \mathbb{N}}\left(1+\frac{1}{q^k}\right). 
\end{eqnarray}
Since $\log(1+x)\leq x$ for $x\geq 0$, we have
\begin{eqnarray*}
\log\left(\prod_{k\in\mathbb{N}}\left(1+\frac{1}{q^k}\right)\right)
&=&\sum_{k\in\mathbb{N}}\log\left(1+\frac{1}{q^k}\right)\leq\sum_{k\in\mathbb{N}}\frac{1}{q^k}=\frac{1}{q-1}. 
\end{eqnarray*}
Thus $L\leq (q^d/rs)\exp(1/(q-1))$. If $r\geq 2$, then 
\[ 
\frac{\exp(1/(q-1))}{r}\leq \frac{\exp(1/(q-1))}{2}\leq \frac{1}{2}+\frac{1}{q-1}<1+\frac{1}{q-1}=\frac{q}{q-1} 
\] 
(the second inequality follows from the inequality $\exp(y)\leq 1+2y$, which is valid for $0\leq y\leq 1$), and 
hence $L\leq q^{d+1}/s(q-1)$ and the result follows.  

Thus we may assume that $q^{a_\alpha}-\varepsilon_\alpha$ and $q^{a_\beta}-\varepsilon_\beta$ are coprime, for 
distinct $\alpha, \beta\in \{1,\ldots,t\}$. In  particular, $q$ is even and so $s=1$. 
Consider distinct $\alpha, \beta\in I$. A direct computation shows that  $q^{a_\alpha}+1$ and $q^{a_\beta}+1$ 
have a non-trivial common factor if and only if $(a_\alpha)_2=(a_\beta)_2$. Thus in particular, for each $k\geq 0$, 
there is at most one $i\in I$ with $(a_i)_2=2^k$. From~\eqref{eq1}, we have
\begin{eqnarray}\label{eq2}
L&\leq&q^d\prod_{i\in  I}\left(1+\frac{1}{q^{a_i}}\right)\leq
q^d\prod_{k\geq 0}\left(1+\frac{1}{q^{2^k}}\right)
\end{eqnarray}
(where in the last inequality we use the fact that if $2^k = (a_i)_2$, then $1+1/q^{a_i}\leq 1+1/q^{2^k}$). By expanding the infinite product on the right hand side of~\eqref{eq2}, we see that
\[ \prod_{k\geq 0}\left(1+\frac{1}{q^{2^k}}\right)=\sum_{r\geq 0}\frac{1}{q^r}=\frac{q}{q-1} \]
and the lemma is proved.
\end{proof}


In the remainder of this section the vector space $V$ admits a non-degenerate form or quadratic 
form of classical type which is preserved up to a scalar multiple by the preimage in $\GL_d(q^\delta)$ 
of the group $G$.  We frequently make use of a theorem of 
B. Huppert~\cite[Satz 2]{H}, which we apply to semisimple elements $\overline{s}\in G$ that preserve the form. 
Such elements generate a subgroup acting completely reducibly on $V$, and by Huppert's Theorem, $V$ 
admits an orthogonal decomposition of the following form which gives finer information than we 
had in Notation~\ref{not1}:
\begin{eqnarray}\label{eq:decomp}
V&=&V_+\perp V_-\perp ((V_{1,1}\oplus V_{1,1}')\perp \cdots \perp (V_{1,m_1}\oplus V_{1,m_1}'))\perp \cdots \\
&&\perp ((V_{r,1}\oplus V_{r,1}')\perp \cdots \perp (V_{r,m_r}\oplus V_{r,m_r}'))\nonumber\\
&&\perp (V_{r+1,1}\perp \cdots \perp V_{r+1,m_{r+1}})
\perp\cdots \perp (V_{t',1}\perp \cdots \perp V_{t',m_{t'}}) \nonumber
\end{eqnarray}
where $V_{+}$ and $V_{-}$ are the  eigenspaces of $\overline{s}$ for the eigenvalues 
$1$ and $-1$, of dimensions $d_+$ and $d_-$, respectively (note $V_\pm$ is non-degenerate if 
$d_\pm>0$ and we set $d_-=0$ if $q$ is even),
and each  $V_{i,j}$ is an irreducible $\mathbb{F}_{q^\delta}\langle \overline{s}\rangle$-submodule.
Moreover for $i=r+1,\ldots,t'$, $V_{i,j}$ is non-degenerate of dimension $2d_i/\delta$  and $\overline{s}$ 
induces an element $y_{i,j}$ of order dividing $q^{d_i}+1$ on $V_{i,j}$ (in the unitary case $\delta=2$ and 
the dimension $d_i$  is odd). 
For $i=1,\ldots,r$, $V_{i,j}$ and $V_{i,j}'$ are totally isotropic  of dimension $d_i/\delta$ (here $d_i$ is even if $\delta=2$), 
$V_{i,j}\oplus V_{i,j}'$ is non-degenerate, and $\overline{s}$ 
induces an element $y_{i,j}$ of order dividing $q^{d_i}-1$ on $V_{i,j}$ while
inducing the adjoint representation $(y_{i,j}^{-1})^{tr}$ on  $V_{i,j}'$ (where 
$x^{tr}$ denotes the transpose of the matrix $x$). For our claims about the 
orders of the $y_{ij}$, we also refer to~\cite{BC,KS} 
for some standard facts on the structure of the maximal tori of the fnite classical groups.  

We denote by $\CSp_{2m}(q)$ the conformal symplectic group, that is, the elements of $\GL_{2m}(q)$ preserving 
a given symplectic form up to a scalar multiple. Also $\PCSp_{2m}(q)$ denotes the projection of $\CSp_{2m}(q)$ 
in $\PGL_{2m}(q)$. From~\cite[Table~$5$, page~xvi]{ATLAS}, we have $|\PCSp_{2m}(q):\PSp_{2m}(q)|=\gcd(2,q-1)$. 
In the rest of this section, by abuse of notation, we write $p^{\lceil\log_p(0)\rceil}=1$.

\begin{lemma}\label{cor:PSp}
$\meo(\PCSp_{2m}(q))\leq q^{m+1}/(q-1)$.
\end{lemma}

\begin{proof}
Using \Cref{cor:PGL} and the fact that $\PCSp_2(q) \cong \PGL_2(q)$, we may assume that $m \ge 2$.
  Let $g$ be an element of $\PCSp_{2m}(q)$ and write $g=su=us$ with $s$ semisimple and $u$ unipotent. We use 
Notation~\ref{not1} for the element $s$. First suppose that $g \in \PSp_{2m}(q)$, and let 
$\overline{g}, \overline{s}, \overline{u}\in\Sp_{2m}(q)$ correspond to $g, s, u$, respectively. Consider the orthogonal $\overline{s}$-invariant 
decomposition of $V$ given by \eqref{eq:decomp} (and note that in this case $\delta=1$).  
Here $V_+$ and $V_{-}$ have even dimension, and we write $2m_+ :=\dim V_+$, $2m_- :=\dim V_-$.
Note that, for $1\leq i\leq r$, $V_{i,j}$ and $V_{i,j}'$ are isomorphic $\mathbb{F}_q\langle 
\overline{s}\rangle$-modules if and only if $y_{i,j}$ acts as the multiplication by $1$ or $-1$ 
on $V_{i,j}$, and by definition of $V_\pm$ this is not the case; thus $V_{i,j}$ and $V_{i,j}'$ 
are non-isomorphic.

Now $m=m_+ + m_- + m_1 d_1 +\dots + m_{t'}d_{t'}$, and by the information from \eqref{eq:decomp}
on the orders of the $y_{i,j}$, and the result in Proposition~\ref{more refined} (using
the notation from Notation~\ref{not1}) about the order of $\overline{u}$,   
we see that the order of $g$ is at most
%
\begin{equation} \label{e:PCSporder}
\begin{split}
\lcm_{i=1}^r\{q^{d_i}-1\}&\cdot\lcm_{i=r+1}^{t'}\{q^{d_i}+1\}\cdot \max\{p^{\lceil \log_p(2m_{\pm})\rceil},p^{\lceil \log_p(m_i)\rceil}\mid i=1,\ldots,t'\}.
\end{split}
\end{equation}
%
Using Lemma~\ref{silly2}, for $i=1,\ldots,r$, we see that by replacing the action of 
$g$ on $(V_{i,1}\oplus V_{i,1}')\oplus \cdots\oplus (V_{i,m_i}\oplus V_{i,m_i}')$ with 
the action given by a semisimple element of order $q^{d_im_i}-1$ (and so having only 
two totally isotropic irreducible $\mathbb{F}_q\langle\overline{s}\rangle$-submodules), 
we obtain an element  $g'$ such that $|g|$ divides $|g'|$ and  $m_i=1$. In particular, 
replacing $g$ by $g'$ if necessary, we may assume that $g=g'$. 
With a similar argument, for those $i\in \{r+1,\ldots,t'\}$ with $m_i$ odd and 
$(p,d_i,m_i,f)\neq (2,1,3,1)$, we may assume that $m_i=1$. 
Also, applying again Lemma~\ref{silly2}, for $i\in \{r+1,\ldots,t'\}$, we may 
assume that if $m_i$ is even, then $(d_i,m_i,f)=(1,2,1)$. 

Suppose that, for some $i_0\in \{r+1,\ldots,t'\}$, we have $(p,d_{i_0},m_{i_0},f)=(2,1,3,1)$. 
The element $g$ induces on $W:=V_{i_0,1}\perp V_{i_0,2}\perp V_{i_0,3}$ an element of order 
dividing $(q+1)p^{\lceil \log_p(3)\rceil}=2^2\cdot 3$. 
Let $g'$ be the element acting as $g$ on $W^\perp$, inducing an element of order $q+1$ 
on $V_{i_0,1}$ and inducing a regular unipotent element on $V_{i_0,2}\perp V_{i_0,3}$. 
Now, $g'$ induces on $W$ an element of order $(q+1)p^{\lceil\log_p(4)\rceil}=2^2\cdot 3$. 
Therefore $|g|=|g'|$ and so, we may replace $g$ by $g'$ (note that in doing so the dimension 
of $V_+$ increases by $2$ and $m_{i_0}$ decreases from 3 to 1). 
In particular, we may assume that $m_i=1$ for each $i\in \{r+1,\ldots,t'\}$ with $m_i$ odd.

Suppose that, for some $i_0\in \{r+1,\ldots,t'\}$, we have $(d_{i_0},m_{i_0},f)=(1,2,1)$. 
The element $g$ induces on $W=V_{i_0,1}\perp V_{i_0,2}$ an element of order dividing 
$(p+1)p^{\lceil \log_p(2)\rceil}=(p+1)p$. Let $g'$ be the element acting as $g$ on 
$W^\perp$, inducing an element of order $p+1$ on $V_{i_0,1}$ and inducing an element 
of order $p$ on $V_{i_0,2}$. Now, $g'$ induces on $W$ an element of order $(p+1)p$. 
Therefore $|g|=|g'|$ and so, replacing $g$ by $g'$ if necessary, we may assume that 
$m_i=1$, for each $i\in \{r+1,\ldots,t'\}$. Thus $m=m_+ + m_- + d_1 +\dots + d_{t'}$.

Now, using Lemma~\ref{tori}, we see that the element $g$  has order at most 
\begin{eqnarray}  \label{e:b}
&&
\lcm_{i=1}^r\{q^{d_i}-1\}\cdot\lcm_{i=r+1}^{t'}\{q^{d_i}+1\}\cdot\max\{p^{\lceil \log_p(2m_+)\rceil},p^{\lceil \log_p(2m_-)\rceil}\}\\
&\leq& 
\frac{q^{m+1-m_+-m_-}}{q-1}\max\{p^{\lceil \log_p(2m_+)\rceil},p^{\lceil \log_p(2m_-)\rceil}\}\leq \frac{q^{m+1}}{q-1} \notag
\end{eqnarray}
(where the last inequality follows from an easy computation). This proves the result for elements 
$g\in \PSp_{2m}(q)$.  If $q$ is even then $\PCSp_{2m}(q) = \PSp_{2m}(q)$, and the proof is complete.
Thus we may assume that $q$ is odd, and in this case, by Lemma~\ref{tori}, 
the upper bound is reduced to $q^{m+1}/(2(q-1))$ if $t'\geq2$.

We must consider elements $g\in \PCSp_{2m}(q)\setminus\PSp_{2m}(q)$.
Now $g^{2} \in \PSp_{2m}(q)$ and we have just shown that $|g^2| \le q^{m+1}/(2(q-1))$ if the
parameter $t'$ for $g^2$ is at least $2$, and hence in this case $|g|\leq q^{m+1}/(q-1)$.
Thus we may assume that $t'\in\{0,1\}$.
If $t'=0$ then 
\[|g^2| \le \max\{p^{\lceil \log_p(2m_+)\rceil},p^{\lceil \log_p(2m_-)\rceil}\}\leq p^{\lceil \log_p(2m)\rceil} \le q^{m+1}/2(q-1),\]
 where the last inequality holds unless $(m,q)=(2,3)$ (this follows from a direct computation). We verify directly the claim of the lemma for $\PCSp_4(3)$. Therefore we may assume 
that the parameter $t'=1$ for $g^2$.

In this case the parameters for $g^2$ satisfy $m=m_+ + m_- + d_1$. If $m_+=m_-=0$ then 
$\overline{g}^2$ is semisimple with eigenvalues $\lambda, \lambda^{-1}, \lambda^{q}, \lambda^{-q}, 
\ldots ,\lambda^{q^{m-1}}, \lambda^{-q^{m-1}}$, where $\lambda^{q^{m}\pm 1}=1$. In particular, 
$\overline{g}^{q^{m}\pm 1} = \pm I_{2m}$ and so $g$ has order at most $q^{m}+1$, which is less 
than $q^{m+1}/(q-1)$. Thus we may assume that $m_++m_->0$. Now \eqref{e:b} gives 
$|g^2|\leq (q^{d_1}+1) \max\{p^{\lceil \log_p(2m_+)\rceil},p^{\lceil \log_p(2m_-)\rceil}\}$. 
To bound the right hand side, we may assume that $m_-=0$ and $m=d_1+m_+$.  A direct computation 
shows that, since $q$ is odd, this bound is less than $q^{m+1}/2(q-1)$ (and hence $|g|\leq 
q^{m+1}/(q-1)$) when $m_+ \ge 2$ unless 
$(q,m_+)=(3,2)$ and $g^2$ has order $9(3^{m-2}+1)$.
 If $m_+=1$ then either $\overline{g}^2$ is semisimple and has order at most $q^{m-1}+1$, 
which is less than $q^{m+1}/2(q-1)$, or $\overline{g}^{2}=J_2+h$ where $h$ has order dividing 
$q^{m-1}\pm 1$. The eigenvalues of $\overline{g}^2$ are therefore 
 $\lambda_1, \ldots ,\lambda_{2m-2}$, with each $\lambda_i \ne \pm 1$ and all distinct,
  and $1$ with algebraic multiplicity $2$. The eigenvalues of $\overline{g}$ 
are therefore $a$, $a$, $\nu_1$, 
$\ldots$, $\nu_{2m-2}$ where $a= \pm 1$ and each $\nu_i^2= \lambda_i$; and since 
$\overline{g}$ is not semisimple, the eigenvalue $a$ must have algebraic multiplicity $2$. 
However $\overline{g}$ is a 
similarity  with respect to the skew-symmetric form $J$; that is $\overline{g}^{T}J\overline{g} = 
\mu J$ for some $\mu \in \mathbb{F}_{q}$ and therefore $J^{-1}\overline{g}^{T} J = \mu \overline{g}^{-1}$. 
In particular, $\overline{g}$ and $\mu \overline{g}^{-1}$ 
are $\GL_n(q)$-conjugate and have the same eigenvalues with the same algebraic multiplicities. So since 
$a$ is an eigenvalue of $\overline{g}$ with algebraic multiplicity $2$,  so is $a \mu$ and we must have 
$\mu=1$. But then $g \in \PSp_{2m}(q)$, contradicting our assumption.
 Finally suppose that $(q,m_+)=(3,2)$ and $g^2$ has order $9(3^{m-2}+1)$. 
Then the eigenvalues of $\overline{g}^{2}$ are $1,\lambda_1,\ldots, \lambda_{2m-4}$, 
where $1$ has algebraic multiplicity $4$, the $\lambda_i$ are distinct and $\lambda_i 
\ne \pm 1$. It follows that the eigenvalues of $\overline{g}$  are $a, \nu_1,\ldots, \nu_{2m-4}$,  
where $a =\pm 1$ has algebraic multiplicity $4$, and each $\nu_i^{2} = \lambda_i$ (since $9$ divides $|g|$). 
Again, since  $\overline{g}^{T}J\overline{g} = \mu J$,  it follows that $a \mu$ is also an eigenvalue 
of $\overline{g}$ with algebraic multiplicity $4$, and therefore $\mu =1$ and $g \in \PSp_{2m}(q)$, 
which is a contradiction.
\end{proof}

\begin{remark}\label{rem:2}{\rm
We note that Corollary~\ref{cor:PSp} is, for $q$ even, asymptotically the best possible. 
Indeed, let  $q$ be a $2$-power, let $k$ be a positive integer and let $s$ be a semisimple 
element of  $\PSp_{2^{k+1}-2}(q)\cong\Sp_{2^{k+1}-2}(q)$. 
Suppose that the natural $\mathbb{F}_q\langle\overline{s}
\rangle$-module $V$ decomposes as $V_1\perp \cdots \perp V_k$ with $\dim_{\mathbb{F}_q} V_i=2^i$ 
and with $\overline{s}$ inducing on $V_i$ an element of order $q^{2^{i-1}}+1$. 
(This is the decomposition of \eqref{eq:decomp} for $\overline s$ where we have 
$V_\pm=0, r=0, t'=k$ and for each $i$, $m_i=1, d_i=i$.) 
Now, we have  
\begin{eqnarray*}
|s|&=&\lcm\{q+1,q^2+1,q^{2^2}+1,\ldots,q^{2^{k-1}}+1\}=(q+1)(q^2+1)\cdots (q^{2^{k-1}}+1)\\
&=&q^{2^k-1}\prod_{i=0}^{k-1}\left( 1+\frac{1}{q^{2^i}}\right),
\end{eqnarray*} 
which approaches $q^{2^{k}}/(q-1)$ as $k$ tends to infinity.

Moreover, the extra care that we used in handling the subspaces $V_+$ and $V_-$ 
in the proof of Corollary~\ref{cor:PSp} may seem ostensibly artificial and unnecessary. 
However we remark  that the maximum order of an element $g$ of $\PSp_{36}(2)$ is 
$2^3\cdot(2+1)\cdot(2^2+1)\cdot (2^4+1)\cdot (2^8+1)$ (see~\cite[p.~$808$]{KS}). 
Such an element $g$ can be chosen to be of the form $su=us$ (with $u$ unipotent and 
$s$ semisimple), where the element $\overline{u}$ fixes a $30$-dimensional subspace pointwise
and acts as a regular unipotent element on a $6$-dimensional subspace $W$, and where 
the element $\overline{s}$ acts trivially on $W$. In particular, this shows that the 
contribution of $V_{+}$ and $V_-$ are sometimes essential in achieving the maximum 
element order of $\PSp_{2m}(q)$.}
\end{remark}

The following result is a consequence of Lemma~\ref{cor:PSp} and results in \cite{KS}.

\begin{corollary}\label{cor:PO}Let $q=p^f$ with $p$ a prime. For $m\geq 3$, we have 
$\meo(\PGO_{2m+1}(q))\leq q^{m+1}/(q-1)$ (with $q$ odd), and for $m\geq 4$ and $\varepsilon\in \{+,-\}$, 
we have $\meo(\PGO_{2m}^\varepsilon(q))\leq q^{m+1}/(q-1)$.
\end{corollary}
\begin{proof}
If $q$ is odd, then the result follows by comparing $q^{m+1}/(q-1)$ with the maximum 
element order of the orthogonal groups obtained in~\cite{KS}. Now, assume that $q$ is 
even. It is well-known that orthogonal groups of characteristic $2$ are subgroups of 
the symplectic groups, that is, $\PGO_{2m}^\varepsilon(q)\leq \PCSp_{2m}(q)$, for 
$\varepsilon\in \{+,-\}$ (see~\cite[Section~$5$]{BC} or~\cite[Table~$3.5.C$]{KL}). It 
follows from Lemma~\ref{cor:PSp} that $\meo(\PGO_{2m}^\varepsilon(q))\leq q^{m+1}/(q-1)$, for $\varepsilon\in \{+,-\}$.
\end{proof}

The next two lemmas will be used for computing the maximum element order 
for unitary groups.

\begin{lemma}\label{silly3}
Let $(b_1,\ldots,b_t)$ be a partition of $d$ and let $q$ be a prime power. If $t \ge 2$, then 
$\lcm_{i=1}^t\{q^{b_i}-(-1)^{b_i}\}\leq q^{d-1}-(-1)^{d-1}$. Moreover 
$(q^{d}-(-1)^d)/(q+1)\leq q^{d-1}-(-1)^{d-1}$.
\end{lemma}

\begin{proof}
For the first part of the lemma, we argue by induction on $t$. Note that $q+1$ divides 
$q^{b_i}-(-1)^{b_i}$ for each $i\in \{1,\ldots,t\}$. If $t=2$, then 
\[ 
\lcm\{q^{b_1}-(-1)^{b_1},q^{b_2}-(-1)^{b_2}\}\leq 
\frac{(q^{b_1}-(-1)^{b_1})(q^{b_2}-(-1)^{b_2})}{q+1}\leq q^{d-1}-(-1)^{d-1} 
\]
(where the last inequality follows from a direct computation). Assume that $t\ge 3$. 
Now, by induction, $\lcm_{i=1}^{t-1}\{q^{b_i}-(-1)^{b_i}\}\leq q^{d-b_t-1}-(-1)^{d-b_t-1}$.  Therefore
\begin{eqnarray*}
\lcm_{i=1}^t\{q^{b_i}-(-1)^{b_i}\}&\leq& \frac{1}{q+1}\left(\lcm_{i=1}^{t-1}\{q^{b_i}-(-1)^{b_i} \}\right)(q^{b_t}-(-1)^{b_t})\\
&\leq& \frac{(q^{d-b_t-1}-(-1)^{d-b_t-1})(q^{b_t}-(-1)^{b_t})}{q+1}\leq q^{d-1}-(-1)^{d-1}
\end{eqnarray*}
(where the last inequality, as before, follows by a direct computation).
The last part of the lemma is immediate.
\end{proof}

\begin{lemma}\label{silly4}Let $d=d_++d_-+e$ with $d_+,d_-,e\geq 0$ and $d\geq 3$, and let $q=p^f$ with $p$ a prime number and $f\geq 1$. Then 
\[(q^{e-1}-(-1)^{e-1})\max\{p^{\lceil\log_p(d_+)\rceil},p^{\lceil\log_p(d_-)\rceil}\}\leq
\left\{
\begin{array}{ll}
q^{d-1}-1&\textrm{if }d \textrm{ is odd and }q>p,\\
(p^{d-2}+1)p&\textrm{if }d \textrm{ is odd and }q=p,\\
q^{d-1}+1&\textrm{if }d \textrm{ is even and }q>2,\\
2^2(2^{d-3}+1)&\textrm{if }d \textrm{ is even and }q=2.\\
\end{array}
\right.
\]
\end{lemma}
\begin{proof}
Note that $p^{\lceil\log_p(m)\rceil}\leq p^{m-1}$, for every integer $m\geq 1$.
Interchanging $d_-$ and $d_+$ if necessary, we may assume that $d_-\leq d_+$.  If $d_-\geq 1$, then 
\[ (q^{e-1}-(-1)^{e-1})\max\{p^{\lceil\log_p(d_+)\rceil},p^{\lceil\log_p(d_-)\rceil}\}\leq (q^{d-d_+-2}-(-1)^{d-d_+-2})p^{\lceil\log_p(d_+)\rceil} \]
and the lemma follows with an easy computation (the polynomial in $q$ 
on the right-hand side has degree at most $d-3$). Thus we may 
assume that $d_-=0$. 
Now, the rest of the proof follows easily by treating separately the four cases listed. 
\end{proof}

Let $f$ be a unitary form. We consider $\Delta/Z$, where $\Delta$ is the subgroup of 
$\GL_d(q^2)$ preserving  $f$ up to a scalar multiple, and $Z\cong Z_{q^2-1}$ is the centre of $\GL_d(q^2)$. 
We claim that $\Delta= \GU_d(q)Z$, where $\GU_d(q)$ is the subgroup of $\GL_d(q^2)$ 
preserving $f$. To see this, note that, if $g\in \GL_d(q^2)$ maps $f$ to $af$ for 
some $a\in\mathbb{F}_{q^2}^*$, then for all $v,w\in V$, we have
$a f(v,w)^q=a f(w,v)$ (since $f$ is unitary), which equals 
$f(wg,vg) =f(vg,wg)^q = a^qf(v,w)^q$, and hence $a^{q}=a$. 
Thus $a\in\FF_q$, so $a=b^{q+1}$ for some $b \in \mathbb{F}_{q^2}$ and 
$g= b (b^{-1}g)\in \GU_d(q)Z$. This proves the claim and thus 
we have $\Delta/Z\cong \GU_d(q)/(\GU_d(q) \cap Z) = \PGU_d(q)$.
For the unitary groups $\PSU_d(q)$ to be simple and different from 
$\PSL_2(q)$, we require $d\geq3$ and $(d,q)\ne(3,2)$. 

\begin{lemma}\label{cor:PGU}
\[
\meo(\mathrm{PGU}_d(q))=\left\{
\begin{array}{lcl}
q^{d-1}-1&&\textrm{if }d \textrm{ is odd and }q>p,\\
(p^{d-2}+1)p&&\textrm{if }d \textrm{ is odd and }q=p,\\
q^{d-1}+1&&\textrm{if }d \textrm{ is even and }q>2,\\
4(2^{d-3}+1)&&\textrm{if }d \textrm{ is even and }q=2.\\
\end{array}
\right.
\]
\end{lemma}


\begin{proof}
  Let $g$ be an element of $\PGU_{d}(q)$ and write $g=su=us$ with $s$ semisimple and $u$ unipotent. 
If $g=u$ then, by Lemma~\ref{uni}, $|g|\leq p^{\lceil\log_p(d)\rceil}\leq p^{d-1}$ and the result follows. 
Thus we may assume that $s\ne1$.
We use Notation~\ref{not1} for the element $s$ and a corresponding element $\overline{s}\in \GL_d(q^2)$. 
From our remarks above, $\overline{s}=a\overline{r}$ for some $a\in\FF_{q^2}^*$ and $\overline{r}\in 
\GU_d(q)$, and hence the $\overline{r}$-invariant orthogonal decomposition described in \eqref{eq:decomp}
is also $\overline{s}$-invariant. Recall 
that, for $1\leq i\leq r$,  $|y_{ij}|$ divides $q^{d_i}-1$ and $d_i$ is even, while for  
$r< i\leq t'$,  $|y_{ij}|$ divides $q^{d_i}+1$ and $d_i$ is odd (and $t'\geq1$ since $s\ne1$). 
Also the order of $\overline{s}|_{V_\pm}$ 
is 1 if $q$ is even and at most 2 is $q$ is odd, and the dimension 
$d=d_++d_- +d_1m_1+\cdots+d_{t'}m_{t'}$. Thus $|s|$ divides 
$\prod_{i=1}^{t'}(q^{d_i}-(-1)^{d_i})$. Moreover, combining Notation~\ref{not1} and Proposition~\ref{more refined} 
(together with the description of the maximal tori of $\GU_d(q)$~\cite{BC,KS}), we see that the order of $g$ is at most
$$
\lcm_{i=1}^{t'}\{q^{d_i}-(-1)^{d_i}\}\cdot\max\{p^{\lceil \log_p(d_{\pm})\rceil},p^{\lceil \log_p(m_i)\rceil}\mid i=1,\ldots,t'\}.
$$
if $t'>1$, and it is at most
\[ (q^{d_1}-(-1)^{d_1})\cdot\max\{p^{\lceil\log_p(d_{\pm})\rceil},p^{\lceil \log_p(m_1)\rceil}\} \]
if $t'=1$.
%
Using Lemma~\ref{silly2} and arguing exactly as in the proof of Lemma~\ref{cor:PSp}, we see that by replacing  
$g$ if necessary by an element of larger or equal order, we may assume that $m_i=1$ for every 
$i\in \{1,\ldots,t'\}$, with the  
exception of at most two values of $i$ such that $(q,d_i,m_i)=(2,1,3)$ and such that $g$ induces 
an element of order $(q+1)p^{\lceil\log_p(m_i)\rceil}=3 \cdot2^2=12$ on $V_{i,1}\perp V_{i,2}\perp V_{i,3}$.
However, in these exceptional cases we have $q=2$ and the restriction of the element $g$ to 
$V_{i,1}\perp V_{i,2}\perp V_{i,3}$ is an element of $\PGU_3(2)$, modulo scalars, and the   maximum order 
of such elements is $6$ rather than $12$. 
Thus in these cases we have overestimated the order by a factor of $2$; we may replace the
restriction of $g$ to this space by an element inducing an element of order $3$ on $V_{i,1}$
and an element of order 2 on $V_{i,2}\perp V_{i,3}$ (thus increasing the dimension of $V_+$ by 2). 
In this way, even if the exceptional cases occur, we obtain an element attaining the maximum order for which
$m_i=1$ for every $i\in \{1,\ldots,t'\}$. Thus we see that
\[ 
|g| \le 
\begin{cases}
\displaystyle(q^{d-d_+-d_{-}}-(-1)^{d-d_+-d_-})\max\{p^{\lceil\log_p(d_{\pm})\rceil}\} & \text{if $t'=1$;}\\
\lcm_{i=1}^{t'}\{q^{d_i}-(-1)^{d_i}\}\max\{p^{\lceil \log_p(d_\pm)\rceil}\} &\text{if $t' \ge 2$.}
\end{cases}
\] 
  Using Lemma~\ref{silly3}, it follows that in both cases
\begin{eqnarray*}
|g|&\leq &(q^{d-d_+-d_--1}-(-1)^{d-d_+-d_{-}-1})\max\{p^{\lceil \log_p(d_\pm)\rceil}\}
\end{eqnarray*}
and the proof follows in these cases from Lemma~\ref{silly4}. 
%

From the description of the semisimple elements given above it is easy to see 
that $\PGU_d(q)$ contains an element $g$ with $|g|$ achieving the stated value of $\meo(\PGU_d(q))$. For example, when $d$ is odd and $q>p$, it suffices to take 
$g$ a semisimple element of order $q^{d-1}-1$ in the maximal torus of order $(q+1)(q^{d-1}-1)$. 
Similarly, when $d$ is even and $q=2$, it suffices to fix a $3$-dimensional non-degenerate 
subspace $W$ and take $g=su=us$, with $s$ a semisimple element of order $p^{d-3}+1$ on $W^\perp$ 
and $u$ an element of order $4$ on $W$. The other two cases are similar.
\end{proof}

Finally, combining all the results we have obtained for the non-abelian simple classical groups and  Lang's theorem, we are ready to give a proof of Theorem~\ref{cor:PGammaL}.

\begin{table}[!h]
\begin{tabular}{|c|c|c|}\hline
 Simple Group $T$   &   $\meo(\Aut(T))$    &             Remark\\\hline            
    $\PSL_d(q)$     &   $(q^d-1)/(q-1)$    &      $(d,q)\neq (2,4)$, $(3,2)$\\     
                    &         $6$          &            $(d,q)=(2,4)$\\            
                    &         $8$          &         $(d,q)=(3,2)$\\\hline         
    $\PSU_d(q)$     &     $q^{d-1}-1$      & $d$ odd, $q>p$ and $(d,q)\neq (3,4)$\\
                    &         $16$         &            $(d,q)=(3,4)$\\            
                    &    $(p^{d-2}+1)p$    & $d$ odd, $q=p$ and $(d,q)\neq (5,2)$\\
                    &         $24$         &            $(d,q)=(5,2)$\\            
                    &     $q^{d-1}+1$      &          $d$ even and $q>2$\\         
                    &    $4(2^{d-3}+1)$    &       $d$ even and $q=2$\\\hline      
   $\PSp_{2m}(q)$   & $\leq q^{m+1}/(q-1)$ &            $(m,q)\ne (2,2)$\\         
   $\PSp_{4}(2)$    &         $10$         &          $(m,q)=(2,2)$\\\hline        
$\POmega_{2m+1}(q)$ & $\leq q^{m+1}/(q-1)$ &                \\\hline               
$\POmega_{2m}^+(q)$ & $\leq q^{m+1}/(q-1)$ &                \\\hline               
$\POmega_{2m}^-(q)$ & $\leq q^{m+1}/(q-1)$ &                \\\hline               
\end{tabular}
\caption{Maximum element order of $\Aut(T)$ for $T$ a non-abelian simple classical group}\label{newtable}
\end{table}

\begin{theorem}\label{cor:PGammaL}
For a classical simple group $T$ as in column $1$ of Table~$\ref{newtable}$, the value of $\meo(\Aut(T))$ is
as in column $2$ of Table~$\ref{newtable}$.
\end{theorem}
\begin{proof}
As usual, we write $q=p^f$ for some prime $p$. For each of the classical groups $\PGL_d(q)$, $\PCSp_{2m}(q)$, 
$\PGO_{2m+1}(q)$ and $\PGO_{2m}^+(q)$, let $X$ be the corresponding algebraic group over the 
algebraic closure of the finite field $\mathbb{F}_q$. Let $F:X\to X$ be a Lang--Steinberg map for $X$. 
We denote the group of fixed points of $F$ by $X^F(q)$. In particular, $X^F(q)$ is one of the following groups: 
$\PGL_d(q)$ or $\PGU_d(q)$ (when $X$ is of type $\AA_{d-1}$),  $\PGO_{2m+1}(q)$ (when $X$ is of type $\BB_m$), 
$\PCSp_{2m}(q)$ (when $X$ is of type $\CC_m$), a subgroup of index two of $\PGO_{2m}^+(q)$ or  $\PGO_{2m}^-(q)$ 
(when $X$ is of type $\DD_m$; namely $(\GO^{\pm}_{2m}(q)^\circ)/Z(\GO^{\pm}_{2m}(q)^\circ)$ where $\GO^{\pm}_{2m}(q)^\circ$ 
is the subgroup of $\GO^{\pm}_{2m}(q)$ that stabilizes each of the two $\SO^{\pm}_{2m}(q)$-orbits of $m$-dimensional 
totally singular subspaces; see~\cite[p. 39-41]{Carter}). Write $Y=\PGO_{2m}^+(q)$ or  $\PGO_{2m}^-(q)$, 
as appropriate, 
in these last cases, and in all other cases write $Y=X^F(q)$.
 
Let $T$ be the socle of $X^F(q)$. From~\cite[Table~$5$, page~xvi]{ATLAS}, the automorphism group $A$ of 
$T$ is $(Y\rtimes\langle\phi\rangle). \Gamma$ where $\phi$ is a generator of the group of field automorphisms 
and $\Gamma$ is the group of graph automorphisms of the corresponding Dynkin diagram. In particular,  
$|\Gamma|\in \{1,2,6\}$ and in fact  $|\Gamma|=6$ if and only if $T=\POmega_8^+(q)$. Moreover, $|\Gamma|=2$ 
if and only if $T=\PSL_d(q)$ with $d\geq 3$, $T=\POmega_{2m}^+(q)$ with $m\geq 5$, or $T=\PSp_4(2^{f})$. 

First suppose that $g \in Y \rtimes \langle\phi\rangle$. Then $g=x\psi^{-1}$ with $x\in Y$, 
where $\psi$ is an element of order $e$ in $\langle\phi\rangle$. We have $|\langle\phi\rangle|=2f$ 
if and only if $Y=\PGU_d(q)$ or $Y=\PGO_{2m}^-(q)$, and $|\langle\phi\rangle|=f$ otherwise 
(see~\cite[Table~$5$, page~xvi]{ATLAS} for example).  

If $\psi=1$, then $g\in Y$ and $|g|$ is at most the bound in Table~\ref{newtable}, by the results in 
Corollaries~\ref{cor:PGL} and~\ref{cor:PO}, and Lemmas~\ref{cor:PSp} and~\ref{cor:PGU}. 
So suppose that $\psi\neq 1$; that is $e\geq 2$. Observe that when $X^F(q)$ is untwisted, $\psi$ is the 
restriction to $X^F(q)$ of the Lang--Steinberg map $\sigma_{q_0}$ (where $q_0^e=q$), which by abuse of 
notation, we also denote by $\psi$. When $X^F = \PGU_d(q)$ or $P(\GO^{-}_{2m}(q)^\circ)$, then 
$F= \sigma_q \tau$, where $\tau$ is a graph automorphism of $X$ induced  from the order $2$ symmetry 
of the Dynkin diagram, and  $\psi$ is the restriction to $X^F(q)$ of the Lang--Steinberg map 
$\sigma_{q_0}\tau$ when $e$ is odd (and where $q_0^e=q$) and $\sigma_{q_0}$ when $e=2k$ is even, (and where  $q_0^k=q$). As  in the untwisted case, by abuse of notation we also denote these maps by $\psi$.

By Lang's theorem, there exists $a$ in the 
algebraic group $X$ such that $aa^{-\psi} =x$. Observe that
$ (x\psi^{-1})^e = xx^\psi\cdots x^{\psi^{e-2}} x^{\psi^{e-1}}$ and  write $z=a^{-1}(x\psi^{-1})^ea$. 
Now observe further that
\begin{eqnarray}\label{eqeqeq1}
z^\psi&=& a^{-\psi} (x^\psi x^{\psi^2}\cdots x^{\psi^{e-1}}x^{\psi^e}) a^{\psi}= 
a^{-\psi} (x^\psi x^{\psi^2}\cdots x^{\psi^{e-1}}x) a^{\psi}\\\nonumber
&=&(a^{-\psi} x^{-1}) (xx^{\psi} \cdots x^{\psi^{e-1}})(xa^{\psi})=
a^{-1} (xx^\psi\cdots x^{\psi^{e-1}}) a = a^{-1}(x\psi^{-1})^ea=z
\end{eqnarray}
and so $z$ is invariant under the Lang--Steinberg map $\psi$. 
It follows that in the untwisted cases $z\in  Y(q^{1/e})$, where $Y(q^{1/e})= 
\PGL_d(q^{1/e}), \PGO_{2m+1}(q^{1/e}), \PCSp_{2m}(q^{1/e})$, $\GO_{2m}^{+}(q^{1/e})^\circ/Z(\GO_{2m}^{+}(q^{1/e})^\circ)$. If $Y$ is twisted and $e$ is odd then $z \in Y(q^{1/e})$ where $Y(q^{1/e})= \PGU_d(q^{1/e}) , \GO_{2m}^{-}(q^{1/e})^{\circ}/Z(\GO_{2m}^{-}(q^{1/e})^{\circ})$. So unless $Y$ is twisted and $e$ is even we have 
\[ 
|g|=|x\psi^{-1}|\leq e|(x\psi^{-1})^e|=e|z|\leq e\meo(Y(q^{1/e})). 
\]
Using the bounds obtained in 
Corollaries~\ref{cor:PGL} and~\ref{cor:PO}, and Lemmas~\ref{cor:PSp} and~\ref{cor:PGU} for 
$\meo(Y(q^{1/e}))$ and $\meo(Y)$, we can show (by a straightforward calculation) that the quantity
$e\meo(Y(q^{1/e}))\leq \meo(Y)$ unless  $Y=X^F(q)=\PGL_2(4)$,
and in this case $|g|\leq 6$ (see line 2 of Table~\ref{newtable}). 
If $Y$ is twisted and $e=2k$ is even, then $z \in \PGL_d(q^{1/k})$ or $\GO_{2m}^{+}(q^{1/k})^\circ/Z(\GO_{2m}^{+}(q^{1/k})^\circ)$ and similar arguments eliminate these cases unless $e=2$ (and $\psi$ induces a graph involution in the terminology of \cite{GLS}). But in this case, we appeal to the element order preserving bijection between $\langle \PGL_n(q),\tau \rangle$ conjugacy classes in the coset $\PGL_n(q)\tau$ and $\langle \PGU_n(q),\tau \rangle$ conjugacy classes in the coset $\PGU_n(q)\tau$. See \cite[Lemmas 2.1--2.3]{gowvin} for details. Thus  the case of $e=2$ and $Y= \PGU_d(q)$ can be covered  by the case of $g=x \tau$ and $Y=\PGL_d(q)$ below. Similarly, by \cite[Lemmas 2.1--2.3]{gowvin} the case  $e=2$ and $Y= \PGO^{-}_{2m}(q)$ is covered by the case of $g=x \tau$, $Y= \PGO^{+}_{2m}(q)$ below.

Thus we assume  that $g \notin Y\rtimes \langle\phi\rangle$ from now on. In particular, 
$T$ is either $\PSL_d(q)$ (with $d\geq 3$), $\PSp_4(2^f)$, or $\POmega_{2m}^+(q)$ (that is, 
$T$ is a simple classical  group  admitting a non-trivial graph automorphism). We deal with 
each of these three cases separately. 

\smallskip
\noindent\textsc{Case $Y=X^F(q)=\PGL_d(q)$.} 
\smallskip

\noindent  We may assume that $g=x\psi^{-1} \tau$, with $x\in X^F(q)$, $\psi$ an element of 
order $e$ in $\langle\phi\rangle$ and $\tau$ the inverse-transpose automorphism. In particular, $d \ge 3$.  

First suppose that $\psi=1$ and  set $y=g^2= xx^{-tr}$, where $x^{tr}$ denotes the transpose 
of the matrix $x$. The possibilities for $y$ are described explicitly in \cite[Theorem 4.2]{FG}:
\begin{enumerate}
\item  if  $\theta(t)^{k}$ is an elementary divisor of $y$, then so is $\bar{\theta}(t)^k$ (and with the same multiplcity), where $\bar{\theta}(t)  =t^{\deg \theta} \theta(1/t)/ \theta(0)$;
\item the elementary divisors $(t-1)^{2k}$ occur with even multiplicity for $k=1,2,\ldots$;
\item if $q$ is odd, the elementary divisors $(t+1)^{2k+1}$ occur with even multiplicity for $k=1,2,\ldots$.
\end{enumerate}
Now $\Sp_{2n}(q)$ contains elements $z$ with elementary divisors satisfying the following properties (see \cite[p. 210]{F2} and \cite[Corollary 5.3]{FG}):
\begin{enumerate}
\item  if  $\theta(t)^{k}$ is an elementary divisor of $z$, then so is $\bar{\theta}(t)^k$ (with the same multiplicity); 
\item the elementary divisors $(t-1)^{2k+1}$ occur with even multiplicity for $k=1,2,\ldots$;
\item the elementary divisors $(t+1)^{2k+1}$ occur with even multiplicity for $k=1,2,\ldots$. 
\end{enumerate}
Thus, either (i) $y$ is conjugate to an element of $\Sp_d(q)$ (and $d$ is even), or (ii) an elementary divisor $(t-1)^{2k+1}$ occurs with odd multiplicity. In  case (i), $|g| \le 2 q^{d/2+1}/(q-1)$ by \Cref{cor:PSp}, which is at most $(q^{d}-1)/(q-1)$ unless $(d,q)=(4,2)$. If (ii) holds then $y$ is conjugate to  $u + y'$ for $u = J_{2k_1+1}+ \cdots +J_{2k_l+1} \in \GL_{d'}(q)$ and  $y' \in \Sp_{d-d'}(q)$; in particular,  
\[|g| \le 2  \max_{i} \{ p^{ \lceil \log_p (2k_i+1)\rceil} \}\meo (\mathrm{Sp}_{d-d'}(q)).\] 
Clearly, to bound the right hand side, it suffices to bound   $p^{ \lceil \log_p (2k+1)\rceil} \meo (\Sp_{d-2k-1}(q))$. For $d=3$, either $k=1$ and $|g|=2|J_3|$ or $k=0$ and $|g| \le 2 \meo(\Sp_2(q))=2q+2$; thus $|g| \le (q^{3}-1)/(q-1)$ unless $q=2$.  If $d \ge 4$, then  by  \Cref{cor:PSp} we have (in case (ii))
\[ |g| \le 2 p^{ \lceil \log_p (2k+1)\rceil} q^{(d-2k+1)/2}\]
which we can check is at most $(q^d-1)/(q-1)$ unless $(d,q)=(4,2),(5,2)$.
The exceptional cases $(d,q)=(3,2)$,$(4,2)$, $(5,2)$ from (i) and (ii) can be dealt with by direct computation, 
and we note that the first case appears in line 3 of Table~\ref{newtable}.

Next, suppose that $\psi$ is a non-trivial element of even order $e$. 
By Lang's theorem, there exists $a$ in the algebraic group $X$ with 
$aa^{-\psi\tau}=x$. Note that since $\psi$ and $\tau$ commute, the 
element $\psi\tau$ has order $e$. Now the same argument as in~\eqref{eqeqeq1} 
shows that $z=a^{-1}g^ea$ is fixed by $\psi \tau$. Therefore $g^e$ is $X$-conjugate 
to an element in $X^{\sigma}(q^{1/e})=\PGU_d(q^{1/e})$ where $\sigma = \tau F^{1/e}$ and 
so $|g| \le e \meo(\PGU_d(q^{1/e}))$. 
Lemma~\ref{cor:PGU} implies that the right hand side is less than $(q^d-1)/(q-1)$ for $d \ge 3$.
	
It remains to consider the case where $\psi \in \langle \phi \rangle$ 
has odd order $e\ge 3$. In this case,  $g^2 \in \mathrm{P\Gamma L}_d(q)$ and 
the argument for field automorphisms applied to $g^2$ shows that $|g| \le 2 
e (q^{d/e}-1)/(q^{1/e}-1)$,  and the right hand side is less than $(q^d-1)/(q-1)$ for $e \ge 3$. 

\smallskip

\noindent \textsc{Case $T=\PSp_4(q)$ with $q=2^f$.}

\smallskip

\noindent The cases where $f=1,2$ can be treated by a direct calculation 
(or with the invaluable help of \texttt{magma}~\cite{magma}).  Thus we may 
assume that $f\geq 3$. We have
$g\not\in X^F(q)\rtimes \langle\phi\rangle$,  
 and we note that  $g^2 \in X^F(q) \rtimes \langle \phi \rangle$. 

First suppose that $g^2 \not\in X^F(q)$. Then $g^2=x'\psi'$, for some $x'\in X^F(q)$ 
and for some field automorphism $\psi'$ of order $e \ge 2$. The same argument as in 
the previous case shows that $|g|=2|g^2|\leq  2e \meo(X^F(q^{1/e}))$. Applying 
Lemma~\ref{cor:PSp} implies that $|g| \le 2e  q^{3/e}/(q^{1/e}-1)$, which is 
bounded above by $q^{3}/(q-1)$ as required.

So we may assume that $g^2 \in X^F(q)$. Since  $g \not\in X^F(q)$, the element 
$g$ projects to an element of order $2$ in $\Out(T)$. Now $\Out(T)$ is cyclic of order $2f$ and
is generated by the extraordinary ``graph automorphism''. In particular, if $f$ were even, 
then $g^2$ would not lie in $X^F(q)$. Hence $f$ is odd. We note that $g^2$ cannot 
have order $q^2- 1$ or $q^2+1$, as in these cases $g^2 \in \cent {\PSp_4(q)}{g^{|g^2|}}$ 
and $g^{|g^2|}$ is an outer involution whose centralizer in 
$\PSp_4(q)$ is isomorphic to ${^2}\BB_2(q)$ by~\cite[(19.5)]{AS}. This is not possible 
since the Suzuki groups do not contain elements of order $q^{2}\pm 1$.  It now follows 
from an analysis of the element orders in $\PSp_4(q)$  that $|g^2| \le (q^{2}+1)/2\leq 
q^3/(2(q-1))$ (see~\eqref{e:PCSporder}). Hence $|g|\leq q^3/(q-1)$.
\smallskip

\noindent\textsc{Case $T=\POmega_{2m}^+(q)$. }

\smallskip   

\noindent We may assume that $g = x \psi^{-1} \tau$, where  $x \in \PGO^{+}_{2m}(q)$, 
$\psi \in \langle \phi \rangle$ (the group of field automorphisms) and $\psi$ 
has order $e \ge 1$, and in this case we let $\tau$ denote a graph automorphism of order $2$ or $3$. 
If $e=1$ and $\tau$ has order $2$ then $g \in \PGO^{+}_{2m}(q)$ and 
Corollary~\ref{cor:PO} applies. 

If $\tau$ has order $2$ and $e \ge 2$ then we consider three cases:  
If $e \ge 4$ and $e$ is even, then $g^2 \in Y.\langle \phi \rangle$ is in the 
$Y$-coset of a field automorphism of order $e/2$. Arguing as above we find that $g^{e}$ is $X$-conjugate 
to an element in $X^{F^{2/e}}(q^{2/e})=P(\GO^{\epsilon'}_{2m}(q^{2/e})^\circ)$ \cite[p. 40]{Carter} 
and  $|g| \le e q^{2(m+1)/e}/(q^{2/e}-1)$
by Corollary~\ref{cor:PO}.  
If $e \ge 3$ and $e$ is odd then $g^2$ is in the $Y$-coset of a field automorphism 
of order $e$  and so $g^{2e}$ is $X$-conjugate to an element in $X^{F^{1/e}}(q^{1/e})= P(\GO^{\epsilon'}_{2m}(q^{1/e})^\circ)$; 
therefore $|g| \le 2e q^{(m+1)/e}/(q^{1/e}-1)$. If $e=2$ then, picking $a \in X$ 
such that $x=a a^{-\psi \tau}$,  we can show that $a^{-1}g^2 a$ is fixed by $\tau \psi$ 
(in the same way as in~\eqref{eqeqeq1}); thus $g^2$ is conjugate to an element of 
$P(\GO^{-}_{2m}(q^{1/2})^\circ)$ \cite[4.9.1(a),(b)]{GLS} and $|g| \le 2q^{(m+1)/2}/(q^{1/2}-1)$.  
In all three cases, a direct calculation shows that the upper bounds we have found 
are less than $q^{m+1}/(q-1)$ for all $q$ and all $m\ge 4$.

Now suppose that $\tau$ has order $3$ so that $m=4$. If $e=1$ then $g \in \POmega_{8}^{+}(q).\Sym(3)$ 
if $q$ is even, and $g \in \POmega_{8}^{+}(q).\Sym(4) = 
\PGO_8^{+}(q).3$ if $q$ is odd (see~\cite[p. 75]{Wilson} for example). 
Since $(2,q-1)^2. \POmega^{+}_{8}(q).\Sym(3)$ is a subgroup of $\FF_4(q)$ (see~\cite[Table 5.1]{LSS}), 
it follows that $|g| \le \meo(\FF_4(q))$  and the bound $|g|\le q^5/(q-1)$ follows from~\cite{KS} 
when $q$ is odd and from~\cite{Shinoda}  when $q$ is even.

Finally, if $\tau$ has order $3$ and $e\ge 2$, then  $g^3 \in Y \rtimes \langle \phi \rangle$. 
If $e \ne 3$ then $g^3$ is in the $Y$-coset of a field automorphism of order $e'$ say, 
where $e' \ge 2$. Therefore  $|g| \le 3e' q^{(m+1)/e'}/(q^{1/e'}-1)$ for some $e' \ge 2$. If 
$e=3$ then, picking $a$ in the algebraic group $X$ such that $x=a a^{-\psi \tau}$, we can show that $a^{-1}g^3a$ is 
fixed by $\tau \psi$; thus $a^{-1}g^3a$ is an element of ${^3}\DD_4(q^{1/3})$  \cite[4.9.1(a),(b)]{GLS}. It follows 
that  $|g| \le 3 \meo({^3}\DD_4(q^{1/3}))$,  which is at most $3(q-1)(q^{1/3}+1)$ 
by~\cite{KS} for $q$ odd,  and  by~\cite[Tables 1.1 and 2.2a]{3d4even} for $q$ even,  unless $q^{1/3}=2$.  
For $q^{1/3}=2$, 
we have $\meo({^3}\DD_4(2))=28$ using~\cite{ATLAS}. In all three cases, a direct computation shows 
that our upper bounds are at most $q^{m+1}/(q-1)$ for all $m \ge 4$, as required.     
\end{proof}

\section{Permutation representations of non-abelian simple groups}\label{md}
In this section  we collect in Table~\ref{TableMDR} some results from the 
literature describing the minimal degree of a permutation
representation of each simple group of Lie type. For
the simple classical 
groups this information is obtained  from~\cite[Table~$5.2.A$]{KL} (which in turn 
came from \cite{coop}) and for
the exceptional groups of Lie type it is obtained
from~\cite{Vav3},~\cite[Theorems~$1$,~$2$ and~$3$]{Vav2}, 
and~\cite[Theorems~$1$,~$2$,~$3$ and~$4$]{Vav1}. We note that the rows
corresponding to the classical groups $\POmega_{2m}^+(q)$ and
$\PSU_{2m}(2)$ in~\cite[Table~$5.2.A$]{KL} are incorrect and our
Table~\ref{TableMDR} takes into account the corrections that were 
brilliantly spotted by  Mazurov and Vasil$'$ev~\cite{Vav4} in 1994.  
\vspace*{2.5cm}

\begin{table}[!h]
\begin{tabular}{|c|c|c|}\hline
          Group           &             Degree of Min. Perm. Repres.            &           Condition \\    \hline \\[-2.2ex]           
       $\PSL_d(q)$        &                 $ \displaystyle \frac{q^d-1}{q-1}$                 &       $(q,d)\neq(2,5),(2,7),$\\       
                          &                                                     &         $(2,9),(2,11),(4,2)$\\        
 $\PSL_2(q)$, $\PSL_4(2)$ &                   $5,7,6,11$, $8$                   &          $q=5,7,9,11$  \\\hline \\[-2.2ex]            
      $\PSp_{2m}(q)$      &                $\displaystyle \frac{q^{2m}-1}{q-1}$               & $m\geq 2$, $q>2$, $(m,q)\neq (2,3)$\\ 
      $\PSp_{2m}(2)$      &                   $2^{m-1}(2^m-1)$                  &              $m\geq 3$\\              
$\PSp_4(2)'$, $\PSp_4(3)$ &                      $6$, $27$                      &                 \\\hline \\[-2.2ex]                 
   $\POmega_{2m+1}(q)$    &                $\displaystyle\frac{q^{2m}-1}{q-1}$               &         $m\geq 3$, $q\geq 5$\\        
   $\POmega_{2m+1}(3)$    &                    $3^m(3^m-1)/2$                   &           $m\geq 3$  \\\hline \\[-2.2ex]              
   $\POmega_{2m}^+(q)$    &           $ \displaystyle \frac{(q^m-1)(q^{m-1}+1)}{q-1}$          &         $m\geq 4$, $q\geq 4$\\        
   $\POmega_{2m}^+(3)$    &                  $3^{m-1}(3^m-1)/2$                 &              $m\geq 4$\\              
   $\POmega_{2m}^+(2)$    &                   $2^{m-1}(2^m-1)$                  &           $m\geq 4$  \\\hline \\[-2.2ex]       
   $\POmega_{2m}^-(q)$    &           $\displaystyle \frac{(q^m+1)(q^{m-1}-1)}{q-1}$          &           $m\geq 4$       \\[3mm]\hline \\[-2.2ex]         
       $\PSU_3(q)$        &                       $q^3+1$                       &              $q\neq 5$\\              
       $\PSU_3(5)$        &                         $50$                        &                   \\                  
       $\PSU_4(q)$        &                    $(q+1)(q^3+1)$                   &                   \\                  
       $\PSU_d(q)$        &   $\displaystyle \frac{(q^d-(-1)^d)(q^{d-1}-(-1)^{d-1})}{q^2-1}$  &        $d\geq 5$, $d$ odd or,\\       
                          &                                                     &        $d$ even and $q\neq 2$\\       
      $\PSU_{2m}(2)$      &                $2^{2m-1}(2^{2m}-1)/3$               &           $m\geq 3$   \\\hline \\[-2.2ex]            
        $\GG_2(q)$        &                 $\displaystyle \frac{q^6-1}{q-1}$                 &                $q>4$\\                
        $\GG_2(3)$        &                        $351$                        &                   \\                  
        $\GG_2(4)$        &                        $416$                        &                   \\                  
        $\FF_4(q)$        &           $\displaystyle\frac{(q^{12}-1)(q^4+1)}{q-1}$           &                   \\                  
        $\EE_6(q)$        &           $\displaystyle\frac{(q^9-1)(q^8+q^4+1)}{q-1}$          &                   \\                  
        $\EE_7(q)$        &        $\displaystyle\frac{(q^{14}-1)(q^9+1)(q^5-1)}{q-1}$       &                   \\                  
        $\EE_8(q)$        & $\displaystyle \frac{(q^{30}-1)(q^{12}+1)(q^{10}+1)(q^6+1)}{q-1}$ &                \\[3mm]\hline \\[-2.2ex]              
      ${^2}\BB_2(q)$      &                       $q^2+1$                       &          $q=2^{f}$, $f$ odd\\         
      ${^2}\GG_2(q)$      &                       $q^3+1$                       &           $q=3^f$, $f$ odd\\          
      ${^3}\DD_4(q)$      &                  $(q^8+q^4+1)(q+1)$                 &                   \\                  
      ${^2}\EE_6(q)$      &      $\displaystyle \frac{(q^{12}-1)(q^6-q^3+1)(q^4+1)}{q-1}$     &                   \\                  
      ${^2}\FF_4(q)$      &                $(q^6+1)(q^3+1)(q+1)$                &            $q=2^f$\\\hline            
\end{tabular}
\caption{Degree of the minimal permutation representations}\label{TableMDR}
\end{table}

\section{Proof of Theorem~\ref{main2}}\label{meo}

In this section, we prove Theorem~\ref{main2} by determining the finite non-abelian simple groups $T$ for which $\meo(\Aut(T)) \ge m(T)/4$.

\begin{proof}[Proof of Theorem~\ref{main2}]Let $T$ be a finite non-abelian simple group and write $o(T)=\meo(\Aut(T))$ and $m(T)$ for the minimal degree of a faithful permutation representation of $T$. 
First, we quickly deal with the cases where $T$ is an alternating group or a sporadic group. Then we may assume that $T$ is a simple group of Lie type, where the situation is more complex. 
If $T=\Alt(m)$ (and $m\geq 5$), then the minimal degree of a permutation representation of $T$ is $m$. Since $\Aut(T)$ contains an element of order $m$, we have $\meo(\Aut(T))\geq m$ and so $T$ is one of the exceptions in the statement of the theorem. Similarly, if $T$ is a sporadic simple group (including the Tits group), then  the proof follows from a case-by-case analysis using~\cite{ATLAS}.

If $T$ is a classical group, then the theorem follows by comparing Table~\ref{newtable} with Table~\ref{TableMDR}. We find that if $o(T)\geq m(T)/4$, then either $T=\PSL_d(q)$ or $T$ belongs to a  short list of exceptions. These exceptions are then analysed using~\texttt{magma}.

Now suppose that $T$ is a finite exceptional group.  As one might expect, we consider the possibilities for the Lie type of $T$ on a case-by-case basis. Complete
information on $m(T)$ is listed in Table~\ref{TableMDR}. We shall use repeatedly the inequalities 
\begin{equation} \label{e:meoaut}
  o(T)\leq \meo(\Out(T))\meo(T)\leq |\Out(T)|\meo(T).
\end{equation}
Detailed information on $|\Out(T)|$ and on the group-structure of $\Out(T)$ can be found in~\cite[Table~$5$, page~xvi]{ATLAS}. 

 When $T$ has odd characteristic, we use the explicit formula
for $\meo(T)$ (see~\cite{KS}) together with~\eqref{e:meoaut} to obtain 
 upper bounds on $o(T)$. These bounds suffice  to show that $o(T)<m(T)/4$ when $T = \EE_6(q)$, ${^2}\EE_6(q)$, $\EE_7(q)$, $\EE_8(q)$, $\FF_4(q)$, $\GG_2(q)$, ${^3}\DD_4(q)$ or ${^2}\GG_2(3^f)$.

Now suppose that $T$ has
even characteristic; in this case there is no known formula for $\meo(T)$. In some cases we  therefore  use \emph{ad hoc} arguments. 

First suppose that  $T={^2}\BB_2(2^{2k+1})$ with $k\geq 1$.
\noindent From~\cite[Table~5, page~xvi]{ATLAS},
we see that $|\Out(T)|= 2k+1$. 
It follows from~\cite{Suzuki} that $\meo(T)=2^{2k+1}+2^{k+1}+1$. In 
particular,  $o(T)\leq(2k+1)(2^{2k+1}+2^{k+1}+1)$ and  $(2k+1)(2^{2k+1}+2^{k+1}+1)< m(T)/4$ in all cases. 

For the other exceptional groups we observe that 
every element $g \in T$ can be
written uniquely as $g=su=us$, with $s$ semisimple and $u$ unipotent. In
particular, \[ |g|=|s||u|\leq  |\smax| |\umax| \] 
where $\smax$ is a semisimple element in $T$ of maximum order and $\umax$ is a unipotent element in $T$ of maximum order.
%
Suppose that $T=\EE_6(2^f)$. By~\cite[Table~5, page~xvi]{ATLAS}, we have $|\Out(T)|=2f (3,2^f-1)$. The description of the maximal tori of $T$ in~\cite[Section~$2.7$]{KS2} implies that the maximum order of a semisimple element of $T$ is at most  $\alpha =(q+1)(q^5-1)/(3,q-1)$. From~\cite[Table 5]{Law} we see that the maximum order of a unipotent element in $\EE_6(q)$ is $16=|\umax|$ when $q$ is even.  Summing up, we have  
\begin{equation} 
  o(T) \le \alpha |\umax|  |\Out(T)|,
\end{equation}
and the right hand side in our case is $32f(2^f+1)(2^{5f}-1)$. 
A direct computation shows that the inequality $32f(2^f+1)(2^{5f}-1)< m(T)/4$ holds for all $f \ge 1$.

This argument works for nearly all of the other exceptional groups in even characteristic. We list these cases in Table~\ref{tab:oTexceptional}. For the reader's convenience we list the formulas for $|\Out(T)|$  in column 4 of Table~\ref{tab:oTexceptional} for all $q$ (not necessarily of the form $q=2^f$). 
For nearly all values of $q=2^f$, we have 
\begin{equation} \label{e:oTupperbound}
  m(T)/4 > \alpha |\umax|  |\Out(T)|;
\end{equation} 
Column 5 of Table~\ref{tab:oTexceptional} lists the only values of $q=2^f$ for which  the inequality in~\eqref{e:oTupperbound} fails.

\begin{table}[htdp]
\begin{center}\begin{tabular}{ccccc}
\hline
             $T$                &   $ \alpha$ where $|\smax| \le \alpha$   &    $|\umax|$     &    $|\Out(T)|$     &             $2^f$ where   \\           
                                &                                          &                  &                    &   \eqref{e:oTupperbound} fails \\\hline \\[-2.2ex]   
        $\EE_6(2^f)$            &       $(2^f+1)(2^{5f}-1)/(3,q-1)$        &      $16$        &    $2f (3,q-1)$    &                  --- \\                
          $\EE_7(2^f)$          &          $(q+1)(q^2+1)(q^4+1)$           &       $32$       &    $f(2,q-1)$      &                  --- \\                
         $\EE_8(2^f)$           &         $(q+1)(q^2+q+1)(q^5-1)$          &       $32$       &        $f$         &                  --- \\                
         $\FF_4(2^f)$           &             $(q+1)(q^3-1)$               &       $16$       &    $ f(2,p)$       &                   ---\\                
   $\GG_2(2^f)$ ($f \ge 2$)     &                $q^2+q+1$                 &         8        &      $f(3,p)$      &                  $4$ \\                
      ${^3}\DD_4(2^{f})$        &              $q^4+q^3-q-1$               &        $8$       &       $3f$         &                  $2$ \\                
      ${^{2}}\EE_6(2^f)$        &      $(q+1)(q^2+1)(q^3-1)/(3,q+1)$       &       $16$       &    $2f(3,q+1)$     &                  --- \\                
${^2}\FF_4(2^{f})$ ($f \ge 3$)  &     $q^2+\sqrt{2q^3}+q+\sqrt{2q}+1$      &       $16$       &       $ f$         &                           ---\\        
\hline
 \end{tabular} \caption{Calculations in proof of Theorem~\ref{main2}}
\end{center}
\label{tab:oTexceptional}
\end{table}
 In view of Column 5  of Table~\ref{tab:oTexceptional}, it remains to consider $T=\GG_2(4)$ and ${^3}\DD_4(2)$.
 In the first case we see from~\cite[page~$97$]{ATLAS} that the maximum element order of $\Aut(\GG_2(4))$ is $24$ and so $24=o(T)<m(T)/4=104$. 
 In the second case we see from~\cite[page~$89$]{ATLAS} that the maximum element order of $\Aut({^3}\DD_4(2))$ is $24$ and so $24=o(T)<m(T)/4=819/4$.
\end{proof}

\section{Proof of Theorem~\ref{main}}\label{sec:them2}

In this section, we classify the primitive permutation groups of degree $n$ that contain an element of order at least $n/4$.
Our proof proceeds according to the O'Nan--Scott type of the primitive permutation 
group $G$, and we use the notation for these types discussed in Subsection~\ref{sub1}.
We treat the almost simple AS  and the simple diagonal SD types in separate 
subsections, and then consider the other types to complete the proof.

\subsection{Proof of Theorem~\ref{main} for almost simple groups}\label{AS}
In this subsection we prove Theorem~\ref{main} for primitive groups
of AS type. We start with a series of very technical  lemmas concerning $\GL_d(q)$ and  
the affine general linear group $\AGL_d(q)$.

\begin{lemma}\label{lemma1}Let $d\geq 2$ and let $K$ be the subgroup of $\GL_d(q)$ containing $\SL_d(q)$ that satisfies $|\GL_d(q):K|=\gcd(d+1,q-1)$. Assume that there exists $H\leq K$ with $|K:H|\leq 8$. Then either $d=2$ and $q\in \{2,3,4,5,7\}$, or $d\in\{3,4\}$ and $q=2$, or $\SL_d(q)\leq H$.
\end{lemma}
\begin{proof}
Write $G=\GL_d(q)$, $S=\SL_d(q)$ and let $Z=Z(S)$. 
Now either $(H\cap S)Z/Z$ equals $S/Z$ or $(H\cap S)Z/Z$ is a proper subgroup of the simple group 
$S/Z\cong \PSL_d(q)$ of index at most $8$. In the former case, since $S$ is a perfect group, we 
find that $S=S'=((H\cap S)Z)'=(H\cap S)'\leq H\cap S\leq H$.
Checking Table~\ref{TableMDR}, we see that in the latter case we must have $d=2$ and $q\in \{2,3,4,5,7,9\}$, or $d\in \{3,4\}$ and $q=2$. If  $d=2$ and $q=9$ then $K=\GL_2(9)$ and we check using \cite{ATLAS} that if $H$ is a subgroup of index at most $8$ in $K$, then $S\leq H$. 
\end{proof}

\begin{lemma}\label{lemma2}
Let $d\geq 2$ and let $K$ be the subgroup of $\AGL_d(q)$ containing $\ASL_d(q)$ that satisfies $|\AGL_d(q):K|=\gcd(d+1,q-1)$. Suppose that  $H\leq K$ satisfies $|K:H|\leq 8$ and $H=\norm {K}H$. Then either $K=H$, or $d=2$ and $q\in \{2,3,4,5,7\}$, or $d\in \{3,4\}$ and $q=2$.
\end{lemma}
\begin{proof}
Write $G=\AGL_d(q)$ and $S=\SL_d(q)$,  and assume that $K>H$.  Let $V$ be the socle 
of $G$. Now $|K/V:HV/V|\leq 8$ and $K/V$ is isomorphic to the subgroup of $\GL_d(q)$ 
containing $\SL_d(q)$ of index $\gcd(d+1,q-1)$. By Lemma~\ref{lemma1}, we see that 
either $d=2$ and $q\in \{2,3,4,5,7\}$, or $d\in\{3,4\}$ and $q=2$, or $SV\subseteq HV$. 
Suppose that $SV\subseteq HV$. Then the group $HV$ acts by conjugation on $V$ as a 
linear group containing $\SL_d(q)$. Therefore either $V \cap H=1$ or $V \cap H =V$. 
In the former case, $8\geq |K:H|\geq |HV:H|=|V:(V\cap H)|=q^d$ and so $(q,d)=(2,2)$ or $(2,3)$. In the latter case, $V\subseteq H$ and hence $VS\leq H$ and $H\unlhd G$. Since $H=\norm K H$, we have $K=H$, contradicting the fact that $K>H$.
\end{proof}

\begin{lemma}\label{lemma3}Let $K$ be the subgroup of $\AGL_1(q)$ of index $\gcd(2,q-1)$. Suppose that $H\leq K$ satisfies $|K:H|\leq 4$ and $H=\norm KH$. Then either $K=H$ or $q=4$.
\end{lemma}
\begin{proof}
Write $G=\AGL_1(q)$ and assume that $K>H$. Let $V$ be the subgroup of $G$ of order $q$. 
Since $|K:H|\leq 4$ and $H=\norm KH$, it follows that $|K:H|=3$ or $4$ and $H$ is a maximal subgroup of $K$. 
If $HV=H$, then 
$V\leq H$ and $H\unlhd G$, which is a contradiction since $H=\norm KH$. Thus $H<HV\leq K$ and hence $K=HV$.

Since $V$ is abelian, we have $V\cap H\unlhd HV=K$. Further, since $V\cap H\leq V$ and 
$K$ acts as a cyclic group of order $(q-1)/\gcd(2,q-1)$ on $V$, it follows that $V\cap 
H=1$ or $V\cap H=V$. In the latter case, $V\leq H$ and $H\unlhd K$, which contradicts 
the fact that $H=\norm K H$. So $V\cap H=1$. Thus $|K:H|=|HV:H|=|V:(V\cap H)|=|V|=q$, so $q\in \{3,4\}$. Finally, it is an easy computation to see that if $q=3$, then $K=V$ and $H$ must be $K$.
\end{proof}

\begin{lemma}\label{lemma:prel}
Let $H$ be a proper subgroup of $T=\PSL_d(q)$ such that $H=\norm TH$ and $|T:H|/4\leq \meo(\Aut(T))$. 
Then one of the following holds:
\begin{enumerate}
\item[(i)] $H$ is conjugate to the stabilizer of a point or a hyperplane of the 
projective space $\mathrm{PG}_{d-1}(q)$;
\item[(ii)] $d=2$ and $q \in \{4,5,7,8,9,11,16,19,25,49\}$, or $d=3$ and $q\in \{2,3,4,5,7\}$, or $d=4$ and $q\in\{2,3\}$, or $d=5$ and $q=2$.
\end{enumerate}
 \end{lemma}
\begin{proof}
Set $q=p^f$, with $p$ a prime and $f\geq 1$. Let $K$ be a maximal subgroup of $T$ 
with $H \le K$. Clearly, $|T:H|\geq |T:K|$ and hence
\begin{equation}\label{QE}
|K|\geq \frac{|T|}{4\meo(\Aut(T))}. 
\end{equation}
In the first part of the proof, we assume that (i) does not hold for the group $K$ and   show that $(d,q)$ must be as in (ii).

First we consider separately  the case that
$d=2$.  We refer to the
description of the lattice of subgroups of $T$  given
in~\cite[Theorem~$6.25$,~$6.26$]{SuzukiBook}. Every subgroup $H$ of $T$ is
either a subgroup of a dihedral group of order $2(q+1)/\gcd(2,q-1)$ or
$2(q-1)/\gcd(2,q-1)$ (if $H$ is as
in~\cite[Theorem~$6.25(a)$]{SuzukiBook}), or a subgroup of a Borel
subgroup of order $(q-1)q/\gcd(2,q-1)$ (if $H$ is as
in~\cite[Theorem~$6.25(b)$]{SuzukiBook}), or isomorphic to $\Alt(4)$,
$\Sym(4)$ or $\Alt(5)$ (if $H$ is as
in~\cite[Theorem~$6.25(c)$]{SuzukiBook}), or isomorphic to
$\PSL_2(q_0)$ or to $\PGL_2(q_0)$ (if $H$ is as
in~\cite[Theorem~$6.25(d)$]{SuzukiBook}, where $q_0$ is a power of $p$ and $q_0^e=q$ for some integer $e$ dividing $f$). Theorem~$6.26$
in~\cite{SuzukiBook}  describes in detail the conditions when each of these
cases can arise. 
For each of the three cases~$(b),(c),(d)$, it can be verified with a tedious computation (using Table~\ref{newtable}) that the inequality $|T:K|/4\leq \meo(\Aut(T))$ is only satisfied if 
$q\in\{4,5,7,8,9,11,16,19,25,49\}$.

We now suppose that $d\geq 3$. Let
$\overline{K}$ be the preimage of $K$ in $\SL_d(q)$ and let
$M$ be a maximal subgroup of $\GL_d(q)$ containing $\overline{K}Z$, where $Z$ is the centre of $\GL_d(q)$. We have $|M|\geq |\overline{K}Z|=(q-1)|K|$. Assume that $|M|<|\GL_d(q)|^{1/3}$. Then~\eqref{QE} implies that
\begin{equation}\label{QE1} 
|\mathrm{GL}_d(q)|^{1/3}>|M|\geq (q-1)|K|\geq\frac{(q-1)|T|}{4\meo(\Aut(T))}.
\end{equation}
A direct computation shows that~\eqref{QE1} is satisfied only if $(d,q)=(3,2)$, 
which is one of the values in (ii). Therefore we may assume that $|\GL_d(q)|^{1/3}
\leq |M|$. Furthermore, for the rest of the proof we assume that $(d,q)\neq (3,2)$ 
and so, according to Table~\ref{newtable}, $\meo(\Aut(T))=(q^d-1)/(q-1)$.

Alavi~\cite[Theorem~$9.1.1$]{Hassan}
classified the maximal subgroups $M$ of $\GL_d(q)$ not containing $\SL_d(q)$ with $|\GL_d(q)|\leq |M|^3$,
listing the possible subgroups according to their ``Aschbacher class'': a detailed description for each
class is given. Using the inequality $|M|\geq (q-1)|K|$, another (rather tedious) computation shows that, for each
of the subgroups listed in~\cite[Theorem~$9.1.1$]{Hassan}  that are not contained in the
Aschbacher class $\mathcal{C}_9$, the
inequality $|T:K|/4\leq (q^d-1)/(q-1)$ is satisfied only in the case that $K$ is conjugate 
to the stabilizer of a point or a hyperplane of $\mathrm{PG}_{d-1}(q)$, or $(d,q)$ is as in (ii). It remains to consider the case that $M$ is contained in the Aschbacher class $\mathcal{C}_9$. In this case, Alavi's classification implies that $d\leq 9$. 

For the rest of the proof of our claim we use Liebeck's result~\cite[Theorem 4.1]{Liebeck11}:  if $H$ is a 
maximal subgroup of $T$ in the Aschbacher class $\mathcal{C}_9$, then either $|H|< q^{3d}$, or 
$H=\Alt(m)$ or $\Sym(m)$ with $m=d+1$ or $d+2$. A straightforward calculation shows that $|\PSL_d(q)|/(4(d+2)!)\leq (q^d-1)/(q-1)$ if and only if $d\in \{3,4\}$ and
$q=2$ or  $(d,q)=(3,3)$. However since $|\PSL_3(3)|$ is not divisible by $d+2=5$, the case $(d,q)=(3,3)$  does not actually occur. In particular, we may assume that $|H|<q^{3d}$. 
Since $|T:H|/4\leq
(q^d-1)/(q-1)$, we have
\begin{eqnarray*}
|T|&\leq&\frac{4(q^d-1)}{q-1}|H|<\frac{4(q^d-1)}{q-1}q^{3d},
\end{eqnarray*} 
which implies that
$d\leq 4$. In particular, we may
assume that $d=3$ or $d=4$. The complete list of the subgroups of $\PSL_3(q)$ and $\PSL_4(q)$ in the Aschbacher class $\mathcal{C}_9$ is contained in Sections~$5.1.2$ and~$5.1.3$ of~\cite{LPS} and in~\cite[Theorem~$1.1$]{Bloom} (for $d=3$ and $q$ odd). A case-by-case analysis now shows that $|T:K|/4 > (q^d-1)/(q-1)$. We have now found all of the values of $(d,q)$ for which (i) does not hold for the group $K$.

Therefore, to conclude the proof we may assume that $K$ is the stabilizer of either 
a point or a hyperplane of $\mathrm{PG}_{d-1}(q)$, and that $H<K$. Now $K$  is isomorphic 
to a subgroup of $\AGL_{d-1}(q)$, namely the subgroup $\tilde{K}$ of $\AGL_{d-1}(q)$ 
containing $\ASL_{d-1}(q)$  that satisfies $|\AGL_{d-1}(q):\tilde{K}|=\gcd(d,q-1)$. 
 Since $H \le T$ and  $H=\norm TH$, we have $H=\norm KH$. Applying Lemma~\ref{lemma2} (for $d\geq 3$) and Lemma~\ref{lemma3} (for $d=2$) implies that $(d,q)=(2,4)$, $d=3$ and $q\in \{2,3,4,5,7\}$, or $d\in \{4,5\}$ and $q=2$. 
\end{proof}

The next proposition is the main ingredient in our proof of Theorem~\ref{main} for projective special linear groups.
\begin{proposition}\label{PSL}
Let $G$ be a primitive group on $\Omega$ of degree $n$ with socle
$\PSL_d(q)$. Assume that the action of $G$ on $\Omega$ is not permutation isomorphic 
to the action on the points or on the hyperplanes of the projective space $\mathrm{PG}_{d-1}(q)$, 
and that  $n/4\leq \meo(\Aut(\PSL_d(q)))$. Then $d=2$ and
$q\in \{4,5,7,8,9,11,16,19,25,49\}$, or $d=3$ and $q\in\{2,3,4\}$, or $d=4$ and $q\in\{2,3\}$.  
\end{proposition}
\begin{proof}
From Table~\ref{newtable} and Lemma~\ref{lemma:prel}, we see that we may assume that $d=2$ and $q\in \{4,5,7,8,9,11,16,19,25,49\}$, or $d=3$ and $q\in\{2,3,4,5,7\}$, or $d=4$ and $q\in \{2,3\}$, or $d=5$ and $q=2$. Now a direct inspection with \texttt{magma}~\cite{magma}, on all the almost simple groups $G$ with socle $T$ and on all maximal subgroups of $G$, shows that only the cases listed in the proposition actually arise.
\end{proof}

For the alternating groups, we will use the following bound in the proof of Theorem~\ref{main}.

\begin{lemma}\label{bound-N}
Let $a,b$ be positive integers, let $m=ab$ and suppose that $a \ge 2$,  $b \ge 2$ and $m \ge 16$. Then
\[\frac{m!}{(a!)^bb!} \ge (1.7)^m.\]
\end{lemma}

\begin{proof}
In~\cite[Lemma 5.13]{BGMRS}, it is proved that if $a\geq 4$, $b\geq 3$ and $m\geq 16$, then $m!/(a!^bb!)\geq (2.2)^m$. In particular, we may assume that $b=2$ or $a\leq 3$. Suppose that $a=2$ or $a=3$.
Stirling's formula~\cite{Rob} implies that for every $m \ge 1$, we have
\[\sqrt{2\pi m}\ e^{1/(12m+1)}(m/e)^m \le m! \le \sqrt{2\pi m}\ e^{1/(12m)}(m/e)^m.\]
Using these inequalities it follows immediately that $2^{m/2}(m/2)!\geq 3!^{m/3}(m/3)!$ for all $m\geq 16$. Thus we have
\begin{eqnarray*}
\frac{m!}{(a!)^bb!}&\geq&\frac{m!}{2^{m/2}(m/2)!}\geq \frac{\sqrt{2\pi m}e^{1/(12m+1)}(m/e)^m}{2^{m/2}\sqrt{\pi m}e^{1/6m}(m/(2e))^{m/2}}\\
&=&\sqrt{2}e^{1/(12m+1)-1/6m}(m/e)^{m/2}\geq 2^m,
\end{eqnarray*}
where the last inequality follows from a direct computation.

Now suppose that $b=2$. We have
\begin{eqnarray*}
\frac{m!}{(m/2)!^22!}&\geq&\frac{\sqrt{2\pi m}e^{1/(12m+1)}(m/e)^m}{2\pi me^{2/6m}(m/(2e))^{m}}\\
&=&\frac{2^m}{\sqrt{2\pi m}}e^{1/(12m+1)-1/3m}\geq (1.7)^m
\end{eqnarray*}
as required (again the last inequality follows from a direct computation).
\end{proof}

\begin{theorem}\label{thm:mainAS}
Let $G$ be a finite primitive group on $\Omega$ of degree $n$ of \emph{AS} type. 
If $G$ contains a permutation $g$ with $|g|\geq n/4$, then the socle $T$ of $G$ 
is either $\Alt(m)$ in its action on the $k$-subsets of $\{1,\ldots,m\}$, for some $k$, or 
$\PSL_d(q)$ in its natural action on the points or on the hyperplanes of the 
projective space $\mathrm{PG}_{d-1}(q)$, or $T$ is one the groups in Table~$\ref{exceptions}$.
\end{theorem}

\begin{proof}
Since all the groups in Table~\ref{exceptionsmain2} are contained in Table~\ref{exceptions}, using Theorem~\ref{main2}, we may assume that $T$ is either an alternating group or a projective special linear group. For $T\cong \PSL_d(q)$, the theorem follows from Proposition~\ref{PSL}.

So we may assume that $T\cong \Alt(m)$ for some $m\geq 5$. Since $\Alt(m)$ is contained in Table~\ref{exceptions} for $m=5,\ldots,9$, we may assume that $m\geq 10$. Now, for  $\omega\in \Omega$, the stabilizer  $G_\omega$ is either intransitive, imprimitive, or primitive in its action on $\{1,\ldots,m\}$. If it is intransitive, then the action of $T$ is permutation equivalent to the action on the $k$-subsets of $\{1,\ldots,m\}$ (for some $1\leq k< m/2$). If  $G_\omega$ is imprimitive in its action on $\{1,\ldots,m\}$, then we can identify the elements of $\Omega$ with the partitions of a set of cardinality $m$ into $b$ parts of cardinality $a$, where $m=ab$ and $a,b\geq 2$. Using Lemma~\ref{bound-N}, if $m\geq 16$, then we have $n=|\Omega|=m!/(a!^bb!)\geq (1.7)^m$. Using this bound and the upper bound for $\meo(\Sym(m))$ in Theorem~\ref{Landau}, we see that the inequality
\[|\Omega|/4\leq \meo(\Sym(m))\]
is never satisfied. For the remaining cases  ($m=11,\ldots,15$) a computation in \texttt{magma}  shows that no examples arise.

Finally, suppose that $G_\omega$ is primitive in its action on $\{1,\ldots,m\}$. In this case, by~\cite{PS}, we have $|G_\omega|\leq 4^m$ and $n=|\Omega|\geq m!/4^m$. Again, using the upper bound in Theorem~\ref{Landau}, we see that  the inequality $|\Omega|/4\leq \meo(\Sym(m))$ is only satisfied for $m\leq 15$. For the remaining cases ($m=11,\ldots,14$) a computation in \texttt{magma}  shows that no examples arise.
\end{proof}

\subsection{Proof of Theorem~\ref{main} for primitive groups of SD type}\label{reduction}

\begin{lemma}\label{another}
Let $T$ be a finite non-abelian simple group. Then $4|\Out(T)|<|T|^{2/3}$.
\end{lemma}
\begin{proof}
The proof follows from a case-by-case analysis (detailed information on $|T|$ and $|\Out(T)|$ can be found in~\cite{ATLAS}).
\end{proof}

\begin{theorem}\label{thm:mainSD}Let $G$ be a finite primitive group on $\Omega$ of degree $n$ of SD type. If $G$ contains a permutation $g$ with $|g|\geq n/4$, then the socle of $G$ is $\Alt(5)^2$ and $|g|=n/4=15$.
\end{theorem}
\begin{proof}
By the description of the O'Nan--Scott types in~\cite{Pr}, there exists a non-abelian 
simple group $T$ such that the socle $N$ of $G$ is  isomorphic to $T_{1}\times \cdots\times T_{\ell}$
with $T_{i} \cong T$ for each $i\in
\{1,\ldots,\ell\}$. The set $\Omega$
can be identified with $T_{1}\times\cdots \times T_{\ell-1}$ and, for the point
$\omega\in \Omega$ that is identified with $(1,\ldots,1)$, the stabilizer 
$N_\omega$ is the diagonal subgroup $\{(t,\ldots,t)\mid t \in T\}$ of $N$. 
That is to say, the action
of $N_\omega$ on $\Omega$ is permutation isomorphic to the action of
$T$ on $T^{\ell-1}$ by ``diagonal'' component-wise conjugation: the
image of the point
$(x_{1},\ldots,x_{\ell-1})$  under the
permutation corresponding to $t\in T$  is
\[ (x_{1}^{t},\ldots,x_{\ell-1}^{t}). \]
The group $G_\omega$ is isomorphic to a subgroup of $\Aut(T)\times \Sym(\ell)$ and $G$ is isomorphic to a subgroup of $T^\ell\cdot(\Out(T)\times \Sym(\ell))$. First suppose that $\ell\geq 3$. Using Lemma~\ref{another}, we have
\begin{eqnarray*}
\meo(G)&\leq &\meo(\Out(T)\times \Sym(l))\meo(T^\ell)\leq |\Out(T)|\meo(\Sym(\ell))\meo(T^\ell)\\
&\leq &|\Out(T)|\meo(\Sym(\ell))|T|
<\meo(\Sym(\ell))(|T|^{5/3}/4).
\end{eqnarray*}
Furthermore, with a direct computation, using Theorem~\ref{Landau} and the fact that $|T|\geq 60$, we can show that $|T|^{\ell-8/3}\geq \meo(\Sym(\ell))$. Thus
\[ \meo(G)<|T|^{\ell-8/3}\frac{|T|^{5/3}}{4}=\frac{|T|^{\ell-1}}{4}=\frac{|\Omega|}{4}. \]

Suppose that $\ell=2$. We claim that $\meo(G)\leq \meo(\Aut(T))^2$. Let $x$ be an element 
of $G$. Now, $x=(g_1,g_2)(1, 2)^i$ for some $i\in \{0,1\}$ where $g_1,g_2\in \Aut(T)$ and 
$g_1\equiv g_2\mod \mathrm{Inn}(T)$. If $i=0$, then $x=(g_1,g_2)$ and $|x|\leq |g_1||g_2|\leq \meo(\Aut(T))^2$. If $i=1$, then 
\[ x^2=(g_1,g_2)(1, 2)(g_1,g_2)(1, 2)=(g_1g_2,g_2g_1). \] 
Now $(g_1g_2)^{g_2^{-1}}=g_2g_1$ and so $|x^2|=|g_1g_2|\leq \meo(\Aut(T))$. Thus $|x|\leq 2\meo(\Aut(T))\leq \meo(\Aut(T))^2$ and our claim is proved. 

Now assume that $T=\Alt(m)$, for some $m\geq 5$. Using Theorem~\ref{Landau}, 
we see that $\meo(\Aut(T))^2<|T|/4$ for every $m\geq 7$. In particular, 
$\meo(G)<|\Omega|/4$, for $m\geq 7$. If $m=6$, then an easy computation shows 
that $\meo(\Alt(6)^2\cdot(\Out(\Alt(6))\times \Sym(2)))=40$ and $|\Omega|=|\Alt(6)|/4
=360/4=90>40$. On the other hand if $m=5$, then 
$|\Omega|/4=|\Alt(5)|/4=60/4=15$ is the order of $(g_1,g_2)\in G$ with $|g_1|=3, |g_2|=5$, 
and this case is in the statement of the theorem. 

Next, suppose that $T=\PSL_d(q)$ for some $m\geq 2$ and $q=p^f$. We may assume that $(m,q)\neq (2,4),(2,5),(2,9)$ and $(4,2)$. Using Table~\ref{newtable}, we find that $\meo(\Aut(T))^2<|T|/4$, for $(m,q)\neq (2,7),(2,8)$ and $(3,2)$. In particular, $\meo(G)<|\Omega|/4$  for $(m,q)\neq (2,7),(2,8)$ and $(3,2)$. Recall that $\PSL_2(7)\cong \PSL_3(2)$. If $(m,q)=(2,7)$, then an easy computation shows that $\meo(\PSL_2(7)^2\cdot(\Out(\PSL_2(7))\times \Sym(2)))=28$ and $|\Omega|=|\PSL_2(7)|/4=168/4=42>28$. Similarly, if $(m,q)=(2,8)$, then $\meo(\PSL_2(8)^2\cdot(\Out(\PSL_2(8))\times \Sym(2)))=63$ and $|\Omega|=|\PSL_2(8)|/4=504/4=126>63$. 

Finally suppose that $T$ is not isomorphic to $\Alt(m)$ or to  $\PSL_d(q)$.  By Theorem~\ref{main2}, it follows that either $\meo(\Aut(T))<m(T)/4$ or that $T$ is one of the groups in Table~\ref{exceptionsmain2}. In the first case, $\meo(\Aut(T))^2<m(T)^2/16\leq |T|/4=|\Omega|/4$ (where the last inequality follows from a direct inspection of Table~\ref{TableMDR}). It remains to suppose that $T$ is one of the groups in Table~\ref{exceptionsmain2}.  Now a case-by-case analysis using~\cite{ATLAS} shows that $\meo(\Aut(T))^2<|T|/4$ in each of the remaining cases.
\end{proof}

\subsection{Proof of Theorem~\ref{main}: the end}\label{reduction2}

We are finally ready to prove Theorem~\ref{main}.  However first we need some more notation. 

\begin{notation}\label{notaPA}{\rm Let  $G$ be a primitive group  of PA or CD type acting 
on $\Omega$. When $G$ is of PA type, the socle  
$\soc(G) = T_1 \times\cdots\times T_\ell$
is isomorphic to $T^\ell$, where $T$ is a non-abelian simple group, and $\ell\geq 2$. When $G$ is of CD type, 
\[
\soc(G)=(T_{1,1}\times \cdots \times T_{1,r})\times \cdots \times (T_{\ell,1}\times \cdots \times T_{\ell,r})
\] 
 is isomorphic to $T^{\ell r}$, where $T$ is a non-abelian simple group and $\ell,r\geq 2$.

In both cases, the action of $G$ on $\Omega$ is permutation isomorphic to the product action of 
$G$ on a set $\Delta^\ell$. By
identifying $\Omega$ with $\Delta^\ell$ we have $G \leq W = H \wr \Sym(\ell)$,  $H \leq \Sym(\Delta)$ 
is primitive on $\Delta$, $\soc(G)$ is the socle of $W$, and $W$ acts on $\Omega$ as in the product 
action. When $G$ is of PA type, $H$ is primitive of AS type and  $\soc(H)=T$. When $G$ is of CD type, 
$H$ is primitive of SD type and $\soc(H)=T^r$ (in particular $|\Delta|=|T|^{r-1}$ and 
$|\Omega|=|T|^{\ell(r-1)}$). 
  
}
\end{notation} 

\begin{proof}[Proof of Theorem~\ref{main}]
Recall that, according to~\cite{Pr}, the finite primitive permutation groups are partitioned into eight families: AS,~HA,~SD,~HS,~HC,~CD,~TW and~PA.  If $G$ is of AS or SD type, then the proof follows from Theorems~\ref{thm:mainAS} and~\ref{thm:mainSD}. If $G$ is of HA type, then the proof follows from~\cite{GMPSaffine}.

Suppose that $G$ is of HS type. Then $G$ is contained in a primitive group $M$ of SD type 
(one might choose $M$ to be $N_{\Sym(n)}(G)$, see~\cite{Pr}).    If $G$ contains an element of order at least $n/4$, then Theorem~\ref{thm:mainSD} implies that the socle of $G$ is $\Alt(5)^2$, which is one of the exceptions listed in Table~\ref{exceptions}.

Next, we recall that every primitive group of TW type is contained in a primitive group of HC type 
(see ~\cite[Section~$4.7$]{DM}), and also every primitive group of HC type is contained 
in a primitive group of CD type (see~\cite{Pr}). Therefore we will assume from now on 
that $G$ is of CD or PA type and we will use Notation~\ref{notaPA}. There are two cases 
to consider: (i)~$H$ contains a permutation $h$ with $|h|> |\Delta|/4$  and  
(ii)~$\meo(H)\leq |\Delta|/4$. Note that Case~(ii) is always satisfied if  $G$ 
is of CD type since, in this case, $H$ is of SD type and Theorem~\ref{thm:mainSD} 
applies. Moreover in Case~(ii) we have
\begin{eqnarray*}
\meo(G)&\leq& \meo(H^\ell)\meo(\Sym(\ell))< (\meo(H))^\ell\meo(\Sym(\ell))\\
&\leq&\frac{|\Delta|^\ell}{4^\ell}\meo(\Sym(\ell))= |\Omega|\frac{\meo(\Sym(\ell))}{4^\ell}\leq \frac{|\Omega|}{4}, 
\end{eqnarray*}
where the second inequality follows since $\ell\geq 2$ and the last inequality follows from 
Theorem~\ref{Landau}. Now suppose that Case~(i) holds; in particular,  $H$ is of AS type. 
By Theorem~\ref{thm:mainAS},   $T=\soc(H)$ is $\Alt(m)$ (in its natural action on $k$-sets) 
or $\PSL_d(q)$ (in its natural action on $\mathrm{PG}_{d-1}(q)$), or $T$ is one of the simple groups in Table~\ref{exceptions}.

It remains to show that there exists a positive integer $\ell_T$ depending only on $T$ with 
$\ell\leq \ell_T$. 
Arguing as above, we have
\begin{eqnarray*}
\meo(G)&\leq& \meo(\Aut(T)^\ell) \meo(\Sym(\ell)) \\ 
&\leq& |\Aut(T)| \meo(\Sym(\ell))\leq |\Aut(T)|  e^{2\sqrt{\ell\log\ell}}
\end{eqnarray*}
where the last inequality follows from Theorem~\ref{Landau}. Since $|\Omega|\geq m(T)^\ell\geq 5^\ell$, it is easy to see that 
$\meo(G)<|\Omega|/4$ for all sufficiently large $\ell$.
\end{proof}

\begin{remark}\label{rm-1}
\emph{In general, the smallest value of $\ell_T$ seems hard to obtain without a 
careful analysis of the element orders of $\Aut(T)$. Nevertheless, for some groups 
$T$ in Table~\ref{exceptions} the number $\ell_T$ can be obtained using some elementary 
arguments. Consider for example the group $T=\Alt(7)$. The element orders of $\Aut(T)
\cong \Sym(7)$ are $1,2,3,4,5,6,7,10$ and $12$. So the maximum element order of 
$\Sym(7)^2$ is $7\cdot 12=84$ and it is not hard to see that the maximum element 
order of $\Sym(7)^\ell$ is $\lcm(7,10,12)=420$ for each integer $\ell\geq 3$. 
In particular, $\meo(\Sym(7)\wr\Sym(\ell))\leq 420\meo(\Sym(\ell))$. 
Now observe that the minimal degree of a permutation representation of 
$\Alt(7)$ is $7$ and $420\meo(\Sym(\ell))<7^\ell/4$ for every $\ell\geq 5$. 
Thus $\ell_T\leq 4$. To obtain the precise value of  $\ell_T$,  one has to embark 
on a careful analysis of the possible element orders of $\Sym(7)\wr\Sym(\ell)$ for $\ell\in \{2,3,4\}$. In this case, it is easy to see that $\ell_T=4$.}

\emph{A similar argument can be used for the Higman--Sims group $T=HS$ for example. 
Remarkably, it turns out that $\ell_T=1$ here, which can be seen using~\cite{ATLAS}.}

\emph{In Table~\ref{anothertable} we give the values of $\ell_T$ for each of the 
simple groups in Table~\ref{exceptions} (these values were obtained with the help 
of a computer). The number $m$ in the table is the degree of the permutation 
representation of the socle factor $T$ of a primitive group $G$ of PA type  
admitting a permutation $g\in G$ with $|g|\geq m^{\ell}/4$.
}
\end{remark}

\begin{table}

\begin{tabular}{|c|c|}\hline
    $T$      &               $(m, \ell_T)$  where $n=m^\ell$ and $1 \le \ell \le \ell_T$ \\\hline             
 $\Alt(5)$   &            $(5,3)$, $(6,3)$, $(10,2)$\\          
 $\Alt(6)$   &           $(6,3)$, $(10,2)$, $(15,1)$\\          
 $\Alt(7)$   &      $(7,4)$, $(15,1)$, $(21,1)$, $(35,1)$\\     
 $\Alt(8)$   & $(8,4)$, $(15,2)$, $(28,1)$, $(35,1)$, $(56,1)$\\
 $\Alt(9)$   &                $(9,4)$, $(36,1)$\\               
  $M_{11}$   &                $(11,3)$, $(12,3)$\\              
  $M_{12}$   &                    $(12,3)$ \\                   
  $M_{22}$   &                    $(22,2)$ \\                   
  $M_{23}$   &                    $(23,3)$ \\                   
  $M_{24}$   &                    $(24,3)$ \\                   
    $HS$     &                    $(100,1)$ \\                  
$\PSL_2(7)$  &       $(7,2)$, $(8,3)$, $(21,1)$, $(28,1)$\\     
$\PSL_2(8)$  &           $(9,2)$, $(28,1)$, $(36,1)$\\          
$\PSL_2(11)$ &                $(11,2)$, $(12,3)$\\              
$\PSL_2(16)$ &                $(17,3)$, $(68,1)$\\              
$\PSL_2(19)$ &                $(20,3)$, $(57,1)$\\              
$\PSL_2(25)$ &                     $(26,2)$\\                   
$\PSL_2(49)$ &                     $(50,2)$\\                   
$\PSL_3(3)$  &               $(13,2)$, $(52,1)$ \\              
$\PSL_3(4)$  &                $(21,2)$, $(56,1)$\\              
$\PSL_4(3)$  &               $(40,2)$, $(130,1)$\\              
$\PSU_3(3)$  &                $(28,1)$, $(36,1)$\\              
$\PSU_3(5)$  &                    $(50,1)$ \\                   
$\PSU_4(3)$  &                    $(112,1)$ \\                  
$\PSp_6(2)$  &                $(28,1)$, $(36,1)$\\              
$\PSp_8(2)$  &                    $(120,1)$ \\                  
$\PSp_4(3)$  &   $(27,1)$, $(36,1)$, $(40,1)$, $(45,1)$\\\hline 
\end{tabular}
\caption{List of degrees $n=m^l$ for which there exists a primitive permutation group
$G$ of degree $n$ as in Theorem~\ref{main}(4)}
\label{anothertable}
\end{table}

\section{Proof of Theorem~\ref{playultimate}}\label{sec:playultimate}

\begin{proof}[Proof of Theorem~\ref{playultimate}]
The first part follows  using the values of $m(T)$ in Table~\ref{TableMDR} and the 
upper bounds on $\meo(\Aut(T))$ in Table~\ref{newtable} in the same way as in  the 
proof of Theorem~\ref{main2}. We only give full details  in the case $T=\PSU_d(q)$, 
with $q\geq 4$. If $d \ge 5$, then $\meo(\Aut(T)) \le q^{d-1}+q^{2}$. So
\[
m(T)^{3/4} = \left(\frac{(q^{d}-(-1)^{d})(q^{d-1} - (-1)^{d-1})}{q^2-1}\right)^{3/4} \ge  (q^{2d-3})^{3/4},
\] 
which is greater than $q^{d-1}+q^2$. If $d=3$, then $m(T)^{3/4} =  (q^{3}+1)^{3/4} >  q^{2}$  
and $\meo(\Aut(T))=q^2-1$ when $q\ne 4$ and so the bound in the statement of 
Theorem~\ref{playultimate} holds with possibly one exception. If $d=4$, then 
$m(T)^{3/4}= (q^{4}+q^{3}+q+1)^{3/4}$ and $\meo(\Aut(T))=q^3+1$ when $q \ne 2$ 
and so the bound in  the statement of Theorem~\ref{playultimate} holds with 
possibly one exception.
Similar calculations show that, apart from a finite number of exceptions,~(i) 
holds for all finite simple groups $T$ satisfying $T \ne \Alt(m)$ and $T\ne \PSL_d(q)$.

To prove the second part of Theorem~\ref{playultimate}, we let $\epsilon,A >0$, 
$g_\epsilon(x)=Ax^{3/4-\epsilon}$ and let $T=\PSU_4(q)$ with $q$ odd. Then 
$\meo(\Aut(T))=q^3+1$ and $m(T) = (q^3+1)(q+1)\le 2q^4$. Thus $g_\epsilon(m(T)) 
\le 2^{3/4} A q^{3-4\epsilon}$, which is strictly less than $q^3+1$ for all  
sufficiently large $q$.
\end{proof} 


\thebibliography{10}
\bibitem{Hassan}S.~H.~Alavi, \textit{Triple factorisations of general linear groups}, PhD  thesis at University of Western Australia, 2010.

\bibitem{AS}
M.~Aschbacher, G.~M. Seitz, Involutions in {C}hevalley groups over fields of even order, {\em Nagoya Math. J.} \textbf{63} (1976), 1--91.

\bibitem{BGMRS}J.~Bamberg, M.~Giudici, J.~Morris, G.~F.~Royle, P.~Spiga, Generalised quadrangles with a group of automorphisms acting primitively on points and lines, \textit{J. Combin. Theory Ser. A} \textbf{119} (2012), 1479--1499.


\bibitem{BT}
A.~D.~Barbour, S.~Tavar\'e, A rate for the Erd\"os--Tur\'an law, \textit{Combin. Probab. Comput.} {\bf3} (1994), 16--176. 

\bibitem{BLNPS}
R.~Beals, C.~R.~Leedham-Green, A.~C.~Niemeyer, C.~E.~Praeger, \'A.~Seress, Permutations with restricted cycle structure and an  algorithmic application, \emph{Combin. Probab. Comput.} {\bf11} (2002), 
447--464.

\bibitem{Bloom}D.~M.~Bloom, The subgroups of $\PSL(3,q)$ for odd $q$, \textit{Trans. Amer. Math. Soc. }\textbf{127} (1967), 150--178.

\bibitem{magma}W.~Bosma, J.~Cannon, C.~Playoust, The Magma algebra system. I. The user language, \textit{J.
Symbolic Comput.} \textbf{24} (1997), 235--265.

\bibitem{BC}A.~A.~Buturlakin, M.~A.~Grechkoseeva, The cyclic structure
  of maximal tori in finite classical groups, \textit{Algebra and
    Logic} \textbf{46} (2007), 73--89.

\bibitem{Carter}
R.~W.~Carter, \textit{Finite groups of {L}ie type}, Wiley Classics Library. John Wiley \& Sons Ltd., Chichester, 1993.
Conjugacy classes and complex characters, Reprint of the 1985
  original, A Wiley-Interscience Publication.

\bibitem{ATLAS}J.~H.~Conway, R.~T.~Curtis, S.~P.~Norton, R.~A.~Parker,
  R.~A.~Wilson, \textit{Atlas of finite groups}, Clarendon Press, Oxford, 1985.

\bibitem{coop} B.~N.~Cooperstein, Minimal degree for a permutation representation of a classical group, 
\textit{Israel J. Math.} {\bf 30} (1978), 213--235.

\bibitem{3d4even}
D.~I.~Deriziotis, G.~O.~Michler, Character table and blocks of finite simple triality groups
  {${^3}\DD_4(q)$}, \textit{Trans. Amer. Math. Soc.} \textbf{303} (1987), 39--70.


\bibitem{DM}J.~Dixon, B.~Mortimer, \textit{Permutation groups}, Springer-Verlag, New York, 1996.


\bibitem{ET1} P. Erd\"os, P. Tur\'an, On some problems of a statistical group-theory, I, 
\emph{Z. Wahrscheinlichkeitstheorie Verw. Gebeite} {\bf 4} (1965), 175--186.

\bibitem{ET2} P. Erd\"os, P. Tur\'an, On some problems of a statistical group-theory, III, 
\emph{Acta Math. Acad. Sci. Hungar.} {\bf 18} (1967), 309--320.

\bibitem{F2}
J.~Fulman, A probabilistic approach to conjugacy classes in the finite
  symplectic and orthogonal groups, \textit{J. Algebra} \textbf{234} (2000), 207--224.

\bibitem{FG}
J.~Fulman, R.~M.~Guralnick, Conjugacy class properties of the
  extension of $\mathrm{GL}_n(q)$ generated by the inverse transpose involution,
  \textit{J. Algebra} \textbf{275} (2004), 356--396.

\bibitem{GLS}
D.~Gorenstein, R.~Lyons, R.~Solomon, \emph{The classification
  of the finite simple groups. number 3. part I. chapter A},  \textbf{40}
  (1998), xvi+419.

\bibitem{gowvin}
R.~Gow, C.~R.~Vinroot, Extending real-valued characters of finite
  general linear and unitary groups on elements related to regular unipotents,
  \textit{J. Group Theory} \textbf{11} (2008), 299--331.

\bibitem{GMPSaffine}
S.~Guest, J.~Morris, C.~E.~Praeger, P.~Spiga, Finite primitve groups of affine type containing elements of large order, in preparation.

\bibitem{GMPScycles}S.~Guest, J.~Morris, C.~E.~Praeger, P.~Spiga, Finite primitive permutation groups containing a permutation with at most four cycles, in preparation.

\bibitem{H}B.~Huppert, Singer-Zylken in klassischen Gruppen, \textit{Math. Z.} \textbf{117} (1970), 141--150.

\bibitem{KS}W.~M.~Kantor, \'{A}.~Seress, Large element orders and the
  characteristic of Lie-type simple groups, \textit{J. Algebra}
  \textbf{322} (2009), 802--832.

\bibitem{KS2}W.~M.~Kantor, \`{A}.~Seress, Prime power graphs for
  groups of Lie type, \textit{J. Algebra} \textbf{247} (2002),
  370--434. 
  
\bibitem{KL}P.~Kleidman, M.~W.~Liebeck, \textit{The subgroup structure of the
  finite classical groups}, London Mathematical Society Lecture Notes
  Series \textbf{129}, Cambridge University Press, Cambridge. 

\bibitem{La1}E.~Landau, \"{U}ber die Maximalordnung der Permutationen gegebenen Grades, \textit{Arch. Math. Phys.} \textbf{5} (1903), 92--103.

\bibitem{La2} E.~Landau, \emph{Handbuch der Lehre vor der Verteilung der Primzahlen}, Teubner, Leipzig, 1909.

\bibitem{Law}R.~Lawther, Jordan block sizes of unipotent elements in exceptional algebraic
  groups, \textit{Comm. Algebra} \textbf{23} (1995), 4125--4156.

\bibitem{Liebeck11}M.~W.~Liebeck, On the orders of maximal subgroups
  of the finite classical groups, \textit{Proc. London Math. Soc. (3)}
  \textbf{50} (1985), 426--446.

\bibitem{LPS1}M.~W.~Liebeck, C.~E.~Praeger, J.~Saxl, On the O'Nan--Scott
  theorem for finite primitive permutation
  groups, \textit{J. Austral. Math. Soc. Ser. A} \textbf{44} (1988),
  389--396.   

\bibitem{LPS}M.~W.~Liebeck, C.~E.~Praeger, J.~Saxl, \textit{The maximal
  factorizations of the finite simple groups and their automorphism
  groups}, Memoirs of the American Mathematical Society, Volume
  \textbf{86}, Nr \textbf{432}, Providence, Rhode Island, USA, 1990. 
  
\bibitem{LM}M.~W.~Liebeck, J.~Saxl, Primitive permutation groups containing an element of large prime order,
\textit{J. London Math. Soc. } \textbf{31} (1985), 237--249. 

\bibitem{LSS}
M.~W.~Liebeck, J.~Saxl,  G.~M.~Seitz, Subgroups of maximal rank
  in finite exceptional groups of lie type, \textit{Proc. London Math. Soc.}
  \textbf{65} (1992), 297--325.

\bibitem{Mass}J.~P.~Massias, J.~L.~Nicolas, G.~Robin, Effective Bounds for the Maximal Order of an Element in the Symmetric Group, \textit{Mathematics of Computation} \textbf{53} (1989), 665--678.

\bibitem{Vav4}V.~D.~Mazurov, A.~V.~Vasil$'$ev,  Minimal permutation representations of finite simple orthogonal groups. (Russian) Algebra i Logika 33 (1994), no. 6, 603--627, 716; translation in Algebra and Logic 33 (1994), no. 6, 337–350 

\bibitem{mueller}P.~M\"{u}ller, Permutation groups with a cyclic two-orbits subgroup and monodromy groups of
Laurent polynomials, \textit{Ann. Scuola Norm. Sup. Pisa} \textbf{12} (2013), to appear.

\bibitem{PS}Cheryl~E.~Praeger, J.~Saxl, On the orders of primitive permutation groups, \textit{Bull. London Math. Soc. }\textbf{12} (1980), 303--307.

\bibitem{Pr}Cheryl~E.~Praeger, Finite quasiprimitive graphs, in Surveys in
  combinatorics, \textit{London Mathematical Society Lecture Note Series}, vol. 24 (1997), 65--85.

\bibitem{Rob}H.~Robbins, A remark on Stirling's formula, \textit{Amer. Math. Monthly} \textbf{62} (1955), 26--29.

\bibitem{Shinoda}
K.~Shinoda, The conjugacy classes of Chevalley groups of type
  ${F}_4$ over finite fields of characteristic $2$, \textit{J. Fac. Sci. Univ. Tokyo Sect. I A Math.} \textbf{21} (1974),
  133--159.

\bibitem{Suzuki}M.~Suzuki, A new type of simple groups of finite
  order, \textit{Proc. Nat. Acad. Sci. U.S.A.} \textbf{46} (1960), 868--870.  

\bibitem{SuzukiBook}M.~Suzuki, \textit{Group Theory I}, Grundlehren
  der mathematischen Wissenschaften \textbf{247}, Springer--Verlag, New
  York, 1981. 
 
\bibitem{Vav3}A.~V.~Vasil$'$ev, Minimal permutation representations of
  finite simple exceptional groups of types $G_2$ and $\FF_4$,
  \textit{Algebra and Logic} \textbf{35} (1996), 371--383.

\bibitem{Vav2}A.~V.~Vasil$'$ev, Minimal permutation representations of
  finite simple exceptional groups of types $\EE_6$, $\EE_7$ and $\EE_8$,
  \textit{Algebra and Logic} \textbf{36} (1997), 302--310.

\bibitem{Vav1}A.~V.~Vasil$'$ev, Minimal permutation representations of
  finite simple exceptional twisted groups, \textit{Algebra and Logic}
  \textbf{37} (1998), 9--20.
  
\bibitem{Wilson}
R.~A.~Wilson, \emph{The finite simple groups}, Graduate Texts in Mathematics, Springer-Verlag (2009),
  xvi+298.
\end{document}